\documentclass[aap,preprint]{imsart}

\RequirePackage[OT1]{fontenc}
\RequirePackage{amsthm,amsmath}
\RequirePackage[numbers]{natbib}
\RequirePackage[colorlinks,citecolor=blue,urlcolor=blue]{hyperref}

\startlocaldefs

\usepackage{mathrsfs}
\usepackage{amssymb}
\usepackage{amsbsy}
\usepackage{bm}

\newtheorem{example}{{\sc Example}}[section]

\newtheorem{theorem}{Theorem}[section]
\newtheorem{proposition}{Proposition}[section]
\newtheorem{assumption}{Assumption}[section]
\newtheorem{corollary}{Corollary}[section]
\newtheorem{remark}{\sc Remark}[section]

\def\pr{\textsf{P}} 
\def\ep{\textsf{E}} 
\def\Cov{\textsf{Cov}} 
\def\Var{\textsf{Var}} 

\endlocaldefs

\begin{document}

\begin{frontmatter}

\title{ Central Limit Theorems of a Recursive Stochastic Algorithm with Applications to Adaptive Designs}
\runtitle{CLT of the Stochastic Algorithm }

\begin{aug}

\author{\fnms{Li-Xin} \snm{ZHANG}\thanksref{t1}\ead[label=e1]{stazlx@zju.edu.cn}}

\address{Li-Xin ZHANG\\
SCHOOL OF MATHEMATICAL SCIENCES$~~~~~~~~~~~~~~~$ \\
ZHEJIANG UNIVERSITY\\
ZHEDA ROAD, NO. 38~~~~~~ \\
HANG ZHOU, 310027~, P.R. China \\
\printead{e1}}


\thankstext{t1}{Research supported by grants from the NSF of China
(No. 11225104), the 973 Program (No. 2015CB352302) and the Fundamental
Research Funds for the Central Universities. }

\runauthor{ L-X. Zhang }
\end{aug}

\begin{abstract}
Stochastic approximation algorithms have been the subject of an enormous body of literature, both theoretical and applied. Recently, Laruelle and Pag\`es (2013) presented a link between the stochastic approximation and   response-adaptive designs in clinical trials based on randomized urn models investigated  in Bai and Hu (1999, 2005), and derived  the asymptotic normality or central limit theorem for the normalized procedure using a central limit theorem for the stochastic approximation algorithm. However, the classical central limit theorem for the stochastic approximation algorithm does not include all cases of its regression function, creating a gap between the results of Laruelle and Pag\`es (2013) and those of Bai and Hu (2005) for randomized urn models. In this paper, we establish  new central limit theorems of the stochastic approximation algorithm under the popular Lindeberg condition to fill this gap. Moreover, we prove that the process of the algorithms can be approximated by a Gaussian process that is a solution of a stochastic differential equation. In our application, we  investigate a more involved family of urn models and related adaptive designs in which it is possible to remove the balls from the urn, and the expectation of the total number of  balls updated at each stage is not necessary a constant. The asymptotic properties are derived under much less stringent assumptions than those in Bai and Hu (1999, 2005) and Laruelle and Pag\`es (2013).

\end{abstract}



\end{frontmatter}



\section{ Introduction}
\setcounter{equation}{0}
Stochastic approximation (SA) algorithms, which have progressively gained sway thanks to the development of computer science and automatic control theory, have been the subject of many studies. An SA algorithm is also used in clinical trials to solve the dose-finding problem (see e.g., Cheung (2010) and the citations therein). The basic frameworks of SA algorithms and their theoretical results can be found in classical textbooks such as those by Benveniste et al. (1990), Duflo (1996, 1997), Kushner and Clark (1978) and Kushner and Yin (2003). In this paper, we consider the following recursive SA algorithm defined on a filtered probability space $\big(\Omega,\mathscr{F}, (\mathscr{F}_n)_{n\ge 0}, \pr)$
\begin{equation}\label{eqModel} \bm\theta_{n+1}=\bm\theta_n-\frac{\bm h(\bm\theta_n)}{n+1}+\frac{\Delta \bm M_{n+1}+\bm r_{n+1}}{n+1},
\end{equation}
where $\bm\theta_n$ is a row vector in $\mathbb R^d$, the regression function $\bm h: \mathbb R^d \to \mathbb R^d$ is a real vector-valued function, $\bm \theta_{0}$ is a finite random vector, $\bm M_0=\bm 0$, $\{\Delta \bm M_n,\mathscr{F}_n;n\ge 1\}$ is a sequence of  martingale differences and $\bm r_n$ is a remainder term.

Very recently, Laruelle and Pag\`es (2013) presented a link between this SA algorithm and the response-adaptive randomization process in clinical trials based on the randomized Generalized Friedman Urn (GFU, also known as a generalized P\'olya urn (GPU) in the literature) models investigated in Bai and Hu (1999, 2005). They derived the almost sure (a.s.) convergence and the joint asymptotic normality or Central Limit Theorem (CLT) of the normalized procedure for both the urn compositions and the assignments by applying SA theory. Higueras et al. (2003, 2006) also showed that the urn compositions can be written as an SA algorithm under some extra assumptions, including that the total number of balls added to the urn at each stage is the same. However, they did not consider the procedure of assignments.

The main tool used by Laruelle and Pag\`es (2013) to derive the asymptotic normality of GPU models is the CLT for an SA algorithm. Various types of results on the CLT of $\bm \theta_n$ have been established in the literature under certain conditions, especially when $\bm r_n\equiv \bm 0$, and they can thus be found in classical textbooks such as that by Kushner and Yin (2003, p. 330). For results in a more general framework, one can refer to Pelletier (1998). Let $\bm \theta^{\ast}$ be an equilibrium point of $\{\bm h=\bm 0\}$. Assume that the function $\bm h$ is differentiable at $\bm \theta^{\ast}$ and that all of the eigenvalues of $D\bm h(\bm \theta^{\ast})=:\big(\partial h_i(\bm\theta^{\ast})/\partial \theta_j;i,j=1,\cdots,d\big)$ have positive real parts. Denote  $\rho=Re(\lambda_{\min})$, where $\lambda_{\min}$  is the eigenvalue of $D\bm h(\bm\theta^{\ast})$ with the lowest real part. In considering the CLT, $\rho>1/2$ is usually assumed as a basic condition. The following CLT can be found in Duflo (1997), Benveniste et al. (1990) and Kushner and Yin (2003) (cf. Theorem A.2  of  Laruelle and Pag\`es (2013)) with different groups of conditions.
\begin{theorem}\label{thLaruellePages} Let $\bm \theta^{\ast}$ be an equilibrium point of $\{\bm h=\bm 0\}$. Suppose that $\bm \theta_n\to \bm\theta^{\ast}$ a.s. and assume that for some $\delta>0$,
\begin{equation}\label{eqLaruellePages1.1} \begin{matrix}
\sup\limits_{n\ge 0} \ep\left[\|\bm\Delta \bm M_{n+1}\|^{2+\delta}\big|\mathscr{F}_n\right]<+\infty\; a.s.\\
\ep\left[(\bm\Delta \bm M_{n+1})^{\rm t}\bm\Delta \bm M_{n+1}\big|\mathscr{F}_n\right]\to \bm\Gamma\; a.s.,
\end{matrix}
\end{equation}
where $\bm \Gamma$ is a deterministic symmetric positive semidefinite matrix and for an $\epsilon>0$,
\begin{equation}\label{eqLaruellePages1.2}
(n+1)\ep\left[\|\bm r_{n+1}\|^2 \mathbb I_{\{\|\bm\theta_n-\bm\theta^{\ast}\|\le \epsilon\}}\right]\to 0.
\end{equation}
Suppose  $\rho:=Re(\lambda_{\min})>1/2$. Then,
     \begin{equation}\label{eqLaruellePages1.3} \sqrt{n}(\bm\theta_n-\bm\theta^{\ast})\overset{\mathscr{D}}\to N\left(\bm 0,\bm\Sigma\right),
     \end{equation}
      where
      \begin{equation}\label{eqLaruellePages1.4} \bm \Sigma:=\int_0^{\infty}\big(e^{-(D\bm h(\bm\theta^{\ast})-\bm I_d/2)u}\big)^{\rm t}\bm\Gamma  e^{-(D\bm h(\bm\theta^{\ast})-\bm I_d/2)u}  du,
      \end{equation}
and $\bm I_d$ is a $d\times d$-identity matrix.
 \end{theorem}
  In the cases of $\rho=1/2$ and $0<\rho<1/2$, partial results have been established when $D\bm h(\bm\theta^{\ast})$ is diagonal. For example,   Duflo (1997, cf. Theorem 2.2.12) showed that if $D\bm h(\bm\theta^{\ast})=\rho\bm I_d$, the CLT holds with rate $\sqrt{\frac{n}{\log n}}$ when $\rho=1/2$, and $n^{\rho}(\bm\theta_n-\bm\theta^{\ast})$ almost surely converges to a random vector when $0<\rho<1/2$.
   Laruelle and Pag\`es (2013) summarized this kind of results to their Theorem A.2 and applied them to GPU, but they missed the condition that $D\bm h(\bm\theta^{\ast})$ is diagonal, thus the results in their Theorems 2.2 (b) and (c)  are not consistent with those in Theorems 2.2 and 3.2 of Bai and Hu (2005). The main purpose of this paper is to establish the CLT for a general matrix $D\bm h(\bm\theta^{\ast})$. We find that in the cases of $\rho=1/2$ and $0<\rho<1/2$, the results for a general matrix $D\bm h(\bm\theta^{\ast})$ are much more complex than those for a diagonal matrix.

In the next section, we establish general  asymptotic results on the SA algorithm (\ref{eqModel}) under the popular Lindeberg condition, which is less restrictive than    (\ref{eqLaruellePages1.1}). From these results, we find that the limiting behavior of the SA algorithm depends on not only the value  of the eigenvalue $\lambda_{\min}$  but also  the  multiplicity of this eigenvalue. Moreover, $n^{\rho}(\bm\theta_n-\bm\theta^{\ast})$  does not converge in general when $\rho<1/2$. Further, in Section \ref{GaussApp}, we prove that the process of the algorithms can be  approximated almost surely by a Gaussian process  when $\rho\le 1/2$  under a condition a little more stringent than the  Lindeberg condition, and the Gaussian process  is a solution of a stochastic differential equation.

As an application of SA theory, in Section  \ref{sectionUrn}, we derive the asymptotic properties of an important class of response-adaptive designs in clinical trials  based on the randomized GFU. Laruelle and Pag\`es (2013) provided a clever way to study the asymptotic normality of randomized  urn models. Motivated by their idea, as an application of the new SA theory, in Section \ref{sectionUrn}, we retrieve the a.s. convergence and the asymptotic normality of the randomized GFU models under assumptions much less stringent than those in Bai and Hu (1999, 2005). We investigate a more involved family of urn models in which it is possible for the balls of each type to be removed from the urn, and the expectation of the total number of balls updated at each stage is not necessarily a constant. The asymptotic property of such urns is stated as an open problem in Hu and Rosenberger (2006, p. 158), and examples of models featuring the removal of balls can be found in Hu and Rosenberger (2006), Janson (2004), Zhang et al. (2011), etc. For this general framework, the first problem is to show the a.s. convergence. The methods of Bai and Hu (1999, 2005) and Higueras et al. (2003, 2006) do  not work  because they depend heavily on the assumption that the total number of balls or the expectation of the total number of  balls updated at each stage is  a constant. We show that the ordinary differential equation (ODE) method proposed by  Laruelle and Pag\`es (2013) is valid to prove the a.s. convergence, although in  their original proof, such an assumption is also needed. However, the ODE is no longer a linear equation, as it was in Laruelle and Pag\`es (2013). The convergence rate of the urn model depends on the second-largest eigenvalue $\lambda_{sec}$ and the largest eigenvalue $\lambda_{max}$ of the urn's limiting generating matrix. When the ratio $\lambda_{sec}/\lambda_{max}$ of these two eigenvalues is large ($>1/2$), the asymptotic property is also an unsolved problem (cf.  Hu and Rosenberger, 2006, p. 158). In section \ref{sectionUrn}, a clear answer to this open problem is provided.

Finally, some basic results on the convergence of the recursive algorithm and multi-dimensional martingales are given in the Appendix.

 In the sequel to this paper, the Euclidean norm of a vector $\bm x=(x_1,\cdots, x_d)$ is defined to be $\|\bm x\|=\sqrt{\sum_j x_j^2}$, and the norm of a matrix $\bm M$ is defined to be $\|\bm M\|=\sup\{\|\bm x\bm M\|:\|\bm x\|=1\}$. $\bm 1=(1,\cdots,1)$ denotes the unit
row vector in $R^d$. $\bm x^{\rm t}$ denotes the transpose of $\bm x$. For a function $\bm f(t): \mathbb R^d  \to \mathbb R$, $\dot{\bm f}(t)$ denotes its derivative, and for a function $\bm f(\bm x): \mathbb R^d \to \mathbb R^d$, $D\bm f(\bm x)$ denotes the matrix of its partial derivatives with the $(i,j)$-th element being $\partial f_i(\bm x)/\partial x_j$. Further, for two positive sequences $\{a_n\}$ and $\{b_n\}$ and a sequence of vectors $\{\bm v_n\}$, we write
$a_n=O(b_n)$ if there is a constant $C$ such that $a_n\le C b_n$,
$a_n\sim b_n$ if $a_n/b_n\to 1$,   $a_n\approx b_n$ if
$a_n=O(b_n)$ and $b_n=O(a_n)$, $\bm v_n=O(b_n)$ if there is a constant $C$ such that $\|\bm v_n\|\le C b_n$, and $\bm v_n=o(b_n)$ if $\|\bm v_n\|/b_n\to 0$.

\section{Central Limit Theorems}
\setcounter{equation}{0}
In this section, we consider the central limit theorem of the SA logarithm (\ref{eqModel}). We first need some assumptions. The first two are on the differentiability of the function $\bm h(\cdot)$.

\begin{assumption} \label{assump0} Let $\bm \theta^{\ast}$ be an equilibrium point of $\{\bm h=\bm 0\}$. Assume that function $\bm h$ is differentiable at $\bm \theta^{\ast}$ and that all of the eigenvalues of $D\bm h(\bm \theta^{\ast})$ have positive real parts.
\end{assumption}
Under Assumption \ref{assump0}, we have that  $\bm h(\bm\theta^{\ast})=\bm 0$,
\begin{equation}\label{eqAssump0.1} \bm h(\bm \theta)=\bm h(\bm \theta^{\ast})+(\bm\theta-\bm\theta^{\ast})D\bm h(\bm \theta^{\ast})
+o\big(\|\bm\theta-\bm\theta^{\ast}\|\big) \;\; \text{ as }\;\; \bm\theta\to \bm\theta^{\ast},
\end{equation}
and   $D\bm h(\bm \theta^{\ast})$ has the following Jordan canonical form
$$ \bm T^{-1} D\bm h(\bm \theta^{\ast})\bm T= diag(\bm J_1, \bm J_2,\cdots,\bm J_s), $$
where
$$\bm J_t=\begin{pmatrix}
\lambda_t &1           &  0     &\ldots     & 0 \\
         0         & \lambda_t  &1       &\ldots     & 0 \\
         \vdots    & \ldots     & \ddots &\ddots     & \vdots \\
         0         & 0          &\ldots  & \lambda_t & 1 \\
         0         & 0          &0       & \ldots    &\lambda_t
\end{pmatrix}_{\nu_t\times\nu_t}=\lambda_t \bm I_{\nu_t}+\overline{\bm J}_{\nu_t}, $$
where $\bm I_{\nu_t}$ is a $\nu_t\times \nu_t$-identity matrix and  $Sp(D\bm h(\bm \theta^{\ast}))=\{\lambda_1,\cdots,\lambda_s\}$  is the set of eigenvalues of $D\bm h(\bm \theta^{\ast})$.
Let $\rho =\min\{Re(\lambda), \lambda\in Sp(D\bm h(\bm \theta^{\ast}))\}$ and $\nu=\max\{\nu_t: Re(\lambda_t)=\rho\}$.

When we consider the case of $\rho\le 1/2$, we need a condition a little more stringent than (\ref{eqAssump0.1}).
\begin{assumption} \label{assump1}
Suppose that Assumption \ref{assump0} is satisfied, $\bm h(\bm\theta^{\ast})=\bm 0$ and
\begin{equation}\label{eqAssump1.1} \bm h(\bm \theta)=\bm h(\bm \theta^{\ast})+(\bm\theta-\bm\theta^{\ast})D\bm h(\bm \theta^{\ast})
+o\big(\|\bm\theta-\bm\theta^{\ast}\|^{1+\epsilon}\big) \;\; \text{ as }\;\; \bm\theta\to \bm\theta^{\ast}
\end{equation}
 for some $\epsilon>0$.
\end{assumption}

We show the CLT under the following conditional Lindeberg's condition, which is popular in the study of the CLT for martingales.

\begin{assumption}\label{assump2}
Suppose that
the following Lindeberg's condition is satisfied:
\begin{equation} \label{eqAssump2.2} \frac{1}{n}\sum_{m=1}^n\ep\left[\|\Delta \bm M_m\|^2\mathbb I\{\|\Delta \bm M_m\|\ge \epsilon \sqrt{n}\}\big|\mathscr{F}_{m-1}\right]\to 0\;\; a.s.\; \text{ or in } L_1,\;\; \forall \epsilon>0.
\end{equation}
Further, assume that
\begin{equation} \label{eqAssump2.3}\frac{1}{n}\sum_{m=1}^n \ep\left[(\Delta \bm M_m)^{\rm t} \Delta \bm M_m\big|\mathscr{F}_{m-1}\right]\to \bm \Gamma
\;\; a.s. \; \text{ or in } L_1,
\end{equation}
where $\bm \Gamma$ is a  symmetric positive semidefinite random matrix.
\end{assumption}

In Assumption \ref{assump2}, $\bm \Gamma$ is a $\mathscr{F}_{\infty}(= \bigvee_n \mathscr{F}_n)$ measurable random matrix, which was assumed to be deterministic in Bai and Hu (1999, 2005), Pelletier (1998) and Laruelle and Pag\`es (2013). Although  $\bm \Gamma$ is usually deterministic in practice, we consider the general martingales, as in Hall and Heyde (1980).
Our main results are the following two theorems on the limiting properties  of the sequence $\{\bm \theta_n\}$ in the cases of $0<\rho<1/2$ and $\rho=1/2$.

\begin{theorem} \label{theorem2} Suppose that $\bm \theta_n\to \bm\theta^{\ast}$ a.s., Assumptions \ref{assump1} and \ref{assump2} are satisfied, and $\rho =1/2$. Further, for the remainder term $\bm r_n$ we assume that
 \begin{equation}\label{conditionTh2.1}
  \sum_{k=1}^n \bm r_k=o(\sqrt{n/\log n}) \; a.s.
  \end{equation}
  or
   \begin{equation}\label{conditionTh2.2}
  \sum_{m=1}^n\frac{\|\bm r_m\|}{\sqrt{m}}=o(\sqrt{\log n}) \;  a.s.\;\; \text{ or in } L_1.
\end{equation}
Then
\begin{equation}\label{eqTh2.1} \frac{\sqrt{n}}{(\log n)^{\nu-1/2}}(\bm\theta_n-\bm \theta^{\ast})  \overset{D}\to N(\bm 0,\widetilde{\bm \Sigma})\;(\text{stably}),
\end{equation}
where
\begin{equation}\label{eqTh2.2} \widetilde{\bm \Sigma}=\lim_{n\to \infty} \frac{1}{(\log n)^{2\nu-1}}\int_0^{\log n}\big(e^{-( D\bm h(\bm \theta^{\ast})-\bm I_d/2)u})^{\rm t}\bm\Gamma e^{-( D\bm h(\bm \theta^{\ast})-\bm I_d/2)u}du,
\end{equation}
and $N(\bm 0,\widetilde{\bm \Sigma})$ denotes a mixing normal distribution with the conditional  characteristic function $f(\bm t)= \exp\big\{-\frac{1}{2}\bm t \widetilde{\bm \Sigma}\bm t^t\big\}$ for given $\widetilde{\bm \Sigma}$.
 Moreover, $\widetilde{\bm\Sigma}$ satisfies
\begin{equation}\label{limitvaraince}
 ({\bm T^{\star}}^{\rm t}\widetilde{\bm \Sigma} \bm T)_{ij} =
\frac 1{((\nu-1)!)^2} \frac 1{2\nu-1}
     \bm t_{a 1}^{\star}\bm \Gamma\bm t_{b 1}^{\rm t},  \end{equation}
whenever $i=\nu_1+ \cdots+\nu_a $, $j=\nu_1+\cdots+\nu_{b}$ and $\lambda_a=\lambda_{b}$, $Re(\lambda_a)=1/2$,
$\nu_a=\nu_b=\nu$, and $ ({\bm T^{\star}}^{\rm t}\widetilde{\bm \Sigma} \bm
T)_{ij}=0$  otherwise. Here, $ \bm x^{\star} $ is the
conjugate vector of a complex   vector $\bm x$ and $\bm t_{a1}^{\rm t}$ is the first column vector of the $a$-th block in $\bm T=[ \cdots,  \bm t_{a1}^{\rm t},\cdots,\bm t_{a\nu_a}^{\rm t},\cdots]$.
 Further, let $\bm r_{a\nu_a}$ be the last row vector of the $a$-th block in $\bm T^{-1}=[\cdots,\bm r_{a1}^{\rm t},\cdots, \bm r_{a\nu_a}^{\rm t},\cdots]^{\rm t}$. Then, $\bm r_{a\nu_a}$ and $\bm t_{a1}^{\rm t}$ are respectively the left and right eigenvectors of $\bm H$ with respect to the eigenvalue $\lambda_a$, and
 $$  \widetilde{\bm \Sigma}  =
\frac 1{((\nu-1)!)^2} \frac 1{2\nu-1} \sum_{a,b: \lambda_a=\lambda_{b}, Re(\lambda_a)=1/2,
\nu_a=\nu_b=\nu}
     (\bm r_{a\nu_a}^{\rm t}\bm t_{a 1})^{\star}\bm \Gamma(\bm t_{b 1}^{\rm t} \bm r_{b\nu_b}). $$
\end{theorem}

\begin{theorem}\label{theorem3} Suppose that $\bm \theta_n\to \bm\theta^{\ast}$ a.s., Assumption \ref{assump1} is satisfied with  $0<\rho<1/2$. Further, assume that
\begin{equation}\label{condition1} \sum_{m=1}^n \ep\left[(\Delta\bm M_m)^{\rm t}\Delta\bm M_m\big|\mathscr{F}_{m-1}\right]=O(n)\; a.s. \;\;\text{ or  in }\; L_1,\; \text{and}
\end{equation}
\begin{equation}\label{condition2} \sum_{k=1}^n \bm r_k=o(n^{1-\rho-\delta_0})\;\; a.s. \text{ for some } \delta_0>0.
\end{equation}
Then, there are complex random variables $\xi_1$, $\cdots$, $\xi_s$ such that
$$ \frac{n^{\rho}}{(\log n)^{\nu-1}}\big(\bm\theta_n-\bm\theta^{\ast}\big)-\sum_{a: Re(\lambda_a)=\rho, \nu_a=\nu} e^{-{\rm i}Im(\lambda_a)\log n} \xi_a\bm e_a\bm T^{-1}\to \bm 0\; a.s., $$
where ${\rm i}=\sqrt{-1}$, $\bm e_a=(\bm 0, \cdots,\bm 0,0,\cdots, 0,1,\bm 0,\cdots, \bm 0)$  is the vector such that the $\nu_a$-th element of its block $a$  is $1$ and other elements are zero, and $\bm e_a\bm T^{-1}=\bm r_{a\nu_a}$ is a right eigenvector of $\bm H$ with respect to the eigenvalue $\lambda_a$.
\end{theorem}

When $\rho>1/2$, the CLT is classical and can be found in the literature under various groups of settings. Moreover, the stepsize $\frac{1}{n+1}$ can be more general. One can refer to Duflo (1997), Benveniste et al. (1990), Kushner and Yin (2003, cf. Theorem 2.1, Chapter 10), Pelletier (1998), etc. Here, in considering applications to a general framework of GPU models, we present the following example under the Lindeberg condition.

\begin{theorem} \label{theorem1}  Suppose that $\bm \theta_n\to \bm\theta^{\ast}$ a.s., Assumptions \ref{assump0} and \ref{assump2} are satisfied, and $\rho >1/2$. Further, for the remainder term $\bm r_n$ we assume that
 \begin{equation}\label{conditionTh1.1}
  \sum_{k=1}^n \bm r_k=o(\sqrt{n}) \; a.s.\;\; \text{ or in } L_1.
  \end{equation}
  Then,
\begin{equation}\label{eqTh1.1} \sqrt{n} (\bm\theta_n-\bm \theta^{\ast})  \overset{D}\to N(\bm 0,\widetilde{\bm \Sigma})\; (\text{stably}),
\end{equation}
where
\begin{equation}\label{eqTh1.2} \widetilde{\bm \Sigma}= \int_0^{\infty}\big(e^{-( D\bm h(\bm \theta^{\ast})-\bm I_d/2)u})^{\rm t}\bm\Gamma e^{-(D\bm h(\bm \theta^{\ast})-\bm I_d/2)u}du.
\end{equation}
\end{theorem}

\begin{remark} The condition (\ref{eqAssump1.1}) cannot be weakened to (\ref{eqAssump0.1}) in Theorems \ref{theorem2} and \ref{theorem3}. The convergence rates in conditions (\ref{conditionTh2.1}) or (\ref{conditionTh2.2}) cannot be weakened in Theorem \ref{theorem2}. For examples see Appendix \ref{sectionexample}.
\end{remark}

From the proof of Theorems (cf. (\ref{eqproofth3.3}) and (\ref{eqproof1.2})), we have the following corollary on the rate of the a.s. convergence.
\begin{corollary}\label{corollary1} Suppose that $\bm \theta_n\to \bm\theta^{\ast}$ a.s., Assumption \ref{assump0} is satisfied with  $\rho>0$. Further assume that condition (\ref{condition1}) in Theorem \ref{theorem3} is satisfied, and
$\frac{1}{n}\sum_{k=1}^n \bm r_k=o(n^{-\frac{1}{2}\wedge \rho+\delta})$  a.s.  for all $\delta>0$.
Then,
$$  \bm\theta_n-\bm\theta^{\ast}=o(n^{-\frac{1}{2}\wedge\rho+\delta})\;\; a.s. \text{ for all } \delta>0. $$
\end{corollary}

 \bigskip
 Now, we give the proof of Theorems \ref{theorem2}-\ref{theorem1}.
 Write $\bm H=D\bm h(\bm\theta^{\ast})$,
 $$ \bm H(\bm\theta)=\begin{cases}
 \bm H+\frac{(\bm \theta-\bm\theta^{\ast})^{\rm t}}{\|\bm\theta-\bm\theta^{\ast}\|}\frac{\bm h(\bm \theta)-\bm h(\bm\theta^{\ast})-(\bm\theta-\bm\theta^{\ast})\bm H}{\|\bm\theta-\bm\theta^{\ast}\|}, &
 \text{ if } \bm \theta\ne \bm\theta^{\ast},\\
 \bm H, & \text{ if } \bm \theta=\bm\theta^{\ast}.
 \end{cases} $$
 Then,   $\bm H(\bm\theta)\to \bm H$ as $\bm\theta\to \bm\theta^{\ast}$ and
$$ \bm h(\bm \theta)=\bm h(\bm \theta^{\ast})+(\bm\theta-\bm\theta^{\ast})\bm H(\bm \theta). $$
Let $\bm H_{n+1}=\bm H(\bm \theta_n)$, $\bm \Pi_n^n=\bm I_d$  and for all $0\le m\le n-1$
\begin{equation}\label{eqproofth1.0} \bm \Pi_m^n  =\Big(\bm I_d-\frac{\bm H_{m+1}}{m+1}\Big)\cdots\Big(\bm I_d-\frac{\bm H_n}{n}\Big),\; \widetilde{\bm \Pi}_m^n=\prod_{j=m+1}^n\Big(\bm I_d-\frac{\bm H}{j}\Big).
\end{equation}
  Then, $\bm H_n\to \bm H$ a.s. as $n\to \infty$. It follows that  for all $1\le m\le n-1$, $\|\bm \Pi_m^n\|\le C_{\delta}\big(n/m)^{-\rho+\delta}$ a.s. and $\|\bm \Pi_0^n\|\le C_{\delta}n^{-\rho+\delta}$ a.s. by Proposition \ref{lemma1} (i) in the Appendix. By (\ref{eqModel}),
\begin{equation}  \bm\theta_{n+1}-\bm\theta^{\ast}=(\bm\theta_n-\bm\theta^{\ast})\left(\bm I_d-\frac{\bm H_{n+1}}{n+1}\right)+\frac{\Delta \bm M_{n+1}+\bm r_{n+1}}{n+1}.
\end{equation}
It follows that
 \begin{equation}\label{eqproofth1.1}
 \bm\theta_n-\bm\theta^{\ast}=(\bm\theta_0-\bm\theta^{\ast})\bm \Pi_0^n+\sum_{m=1}^n \frac{\Delta\bm M_m}{m}\bm \Pi_m^n+\sum_{m=1}^n \frac{\bm r_m}{m}\bm \Pi_m^n.
 \end{equation}
 If we write $\bm s_n=\sum_{m=1}^n \bm r_m$, then the last term is
\begin{align}\sum_{m=1}^n \frac{\bm r_m}{m}\bm \Pi_m^n= &\frac{\bm s_n}{n}\bm \Pi_n^n+\sum_{m=1}^{n-1}\bm s_m\left(\frac{1}{m}\bm \Pi_m^n-\frac{1}{m+1}\bm \Pi_{m+1}^n\right) \nonumber \\
\label{eqproofth1.2}
=&\frac{\bm s_n}{n}\bm \Pi_n^n+\sum_{m=1}^{n-1}\bm s_m\frac{\bm I_d-\bm H_{m+1}}{m(m+1)}\bm \Pi_{m+1}^n.
\end{align}

{\sc Proof of Theorems \ref{theorem2} and \ref{theorem1}.} First we consider the case of $\rho=1/2$.
Suppose the conditions in Theorem \ref{theorem2} are satisfied.
 At first, (\ref{conditionTh2.1}) or (\ref{conditionTh2.2}) will implies that
\begin{equation}\label{eqproofth3.1}
\sum_{m=1}^n \bm r_m=o(n^{1/2+\delta})\; a.s. \;\; \forall \delta>0.
\end{equation}
In fact, it is sufficient to show that (\ref{conditionTh2.2}) implies (\ref{eqproofth3.1}). Note that
 \begin{equation}\label{eqproofth3.1ad}\sum_{m=1}^n \|\bm r_m\|\le \sqrt{n}\sum_{m=1}^n \frac{\|\bm r_m\|}{\sqrt{m}} =o(\sqrt{n\log n})\; a.s. \;\; a.s. \;\text{ or in } \; L_1.
 \end{equation}
Assume that the above inequality holds in the sense of $L_1$. Then,
\begin{align*}
\sum_k \pr \Big(\sum_{m=1}^{2^{k+1}} \|\bm r_m\|\ge \epsilon (2^k)^{1/2+2\delta}\Big)\le C \sum_k\frac{(2^{k+1}\log 2^{k+1})^{1/2}}{(2^k)^{1/2+2\delta}}<\infty.
\end{align*}
It follows that for $2^k\le n\le 2^{k+1}$,
$$\frac{\sum_{m=1}^n \|\bm r_m\|}{n^{1/2+2\delta}}\le \frac{\sum_{m=1}^{2^{k+1}} \|\bm r_m\|}{(2^k)^{1/2+2\delta}}\to 0 \;\; a.s. $$
(\ref{eqproofth3.1}) is true.

In contrast, one can verify that condition (\ref{condition1}) or (\ref{eqAssump2.3}) implies that
\begin{equation}\label{eqproofth3.2}
\bm M_n=o(n^{1/2+\delta})\; a.s. \; \text{ for all } \delta>0.
\end{equation}
Recall  (\ref{eqproofth1.1}), (\ref{eqproofth1.2}) and $\bm H_n\to \bm H$ a.s. as $n\to\infty$. We have
\begin{align}\label{eqproofth3.3}
\bm\theta_n-\bm\theta^{\ast}=&(\bm\theta_0-\bm\theta^{\ast})\bm \Pi_0^n+\frac{\bm M_n+\bm s_n}{n}\bm \Pi_n^n+\sum_{m=1}^{n-1}(\bm M_m+\bm s_m)\frac{\bm I_d-\bm H_{m+1}}{m(m+1)}\bm \Pi_{m+1}^n\nonumber\\
=& o(n^{-\rho+\delta})+\frac{o(n^{1/2+\delta})}{n} +\sum_{m=1}^{n-1}\frac{o(m^{1/2+\delta})}{m(m+1)}\big(\frac{n}{m}\big)^{-\rho+\delta}\nonumber\\
=&o(n^{-1/2+2\delta})\;\; a.s.\;\; \forall \delta>0.
\end{align}
According to condition (\ref{eqAssump1.1}), we have
\begin{align}\label{eqproofth3.5} \bm r_{n+1}^{\ast}=:& (\bm \theta_n-\bm\theta^{\ast})\bm H -\Big(\bm h(\bm\theta_n)-\bm h(\bm\theta^{\ast})\Big)
\nonumber\\
=&  o(\|\bm\theta_n-\bm\theta^{\ast}\|^{1+\epsilon})=o(n^{-1/2-\epsilon/4})\;\; a.s.
\end{align}
 and
$$ \bm\theta_{n+1}-\bm \theta^{\ast}=(\bm\theta_n-\bm \theta^{\ast})\Big(\bm I_d-\frac{\bm H}{n+1}\Big)+\frac{\Delta \bm M_{n+1}+\bm r_{n+1}+\bm r_{n+1}^{\ast}}{n+1}. $$
  It follows that
 \begin{align*}
 \bm\theta_n-\bm\theta^{\ast}=&(\bm\theta_0-\bm\theta^{\ast})\widetilde{\bm \Pi}_0^n+\sum_{m=1}^n \frac{\Delta\bm M_m}{m}\widetilde{\bm \Pi}_m^n+\sum_{m=1}^n \frac{\bm r_m}{m}\widetilde{\bm \Pi}_m^n  +\sum_{m=1}^n \frac{\bm r_m^{\ast}}{m}\widetilde{\bm \Pi}_m^n\nonumber\\
 =:&(\bm\theta_0-\bm\theta^{\ast})\widetilde{\bm \Pi}_0^n+\bm\zeta_n+\bm \eta_n+\bm \eta_n^{\ast},
 \end{align*}
where  $\widetilde{\bm \Pi}_m^n$ is defined as in (\ref{eqproofth1.0}), and $\|\widetilde{\bm\Pi}_m^n\|\le C (n/m)^{-1/2}\log ^{\nu-1}(n/m)$   by Proposition \ref{lemma1} (i) in the Appendix. Thus,
$$ \bm \eta_n^{\ast} =O(1)\sum_{m=1}^n \frac{m^{-1/2-\epsilon/4} }{m}\big(\frac{n}{m}\big)^{-1/2}\log^{\nu-1}\frac{n}{m} =o(n^{-1/2}(\log n)^{\nu-1/2})\; a.s.  $$
  If (\ref{conditionTh2.2}) is satisfied, then
$$ \bm \eta_n =O(1)\sum_{m=1}^n \frac{\|\bm r_m\|}{m}\big(\frac{n}{m}\big)^{-1/2}\log^{\nu-1}\frac{n}{m} =o(n^{-1/2}(\log n)^{\nu-1/2})\; a.s. \text{ or in } L_1. $$
If (\ref{conditionTh2.1}) is satisfied, then we also have
\begin{align*}  \bm \eta_n =&\frac{\bm s_n}{n}  \widetilde{\bm \Pi}_n^n+\sum_{m=1}^{n-1}\bm s_m\frac{\bm I_d-\bm H}{m(m+1)}\widetilde{\bm \Pi}_{m+1}^n\\
=&O(1)\sum_{m=1}^n \frac{o(\sqrt{m/\log m})}{m^2}\big(\frac{n}{m}\big)^{-1/2}\log^{\nu-1}\frac{n}{m} \\
=&O(1)n^{-1/2}\log^{\nu-1} n\sum_{m=1}^n \frac{o(1)}{m\sqrt{\log m}} \\
=&o(n^{-1/2}(\log n)^{\nu-1/2})\; a.s.
\end{align*}
At last, $\bm\zeta_n$ is a  sum of martingale differences. By verifying the Lindeberg condition and checking the variance,  we can show that
$$ \frac{\sqrt{n}}{(\log n)^{\nu-1/2}}\bm\zeta_n \overset{D}\to N(\bm 0,\widetilde{\bm\Sigma})\;(\text{stably}) $$
via the CLT for martingales (cf. Corollary 3.1 of Hall and Heyde (1980)). The above convergence is stated in Proposition \ref{prop2.1} in the Appendix. The proof of Theorem \ref{theorem2} is complete.

Now, we consider the case of $\rho>1/2$. Suppose Assumptions \ref{assump0}, \ref{assump2} and (\ref{conditionTh1.1}) are satisfied.
It is obvious that the first term of (\ref{eqproofth1.1}) is  $O(1)n^{-\rho+\delta}=o(n^{-1/2})$ a.s.,
 and the last term is
\begin{align*}\sum_{m=1}^n \frac{\bm r_m}{m}\bm \Pi_m^n=  \frac{o(\sqrt{n})}{n} +\sum_{m=1}^{n-1}\frac{o(\sqrt{m})}{m(m+1)}\big(\frac{n}{m}\big)^{-\rho+\delta}=o(n^{-1/2})
\end{align*}
in probability by (\ref{conditionTh1.1}) and (\ref{eqproofth1.2}). The middle term of (\ref{eqproofth1.1}) is a sum of weighted martingale differences. Unfortunately, we can not apply the CLT for martingales directly because $\big\{\frac{\Delta\bm M_m}{m} \bm \Pi_m^n; m=1,\cdots, n\big\}$ is not an array of martingale differences. We can show that the random weight
$\bm \Pi_m^n$ can be replaced by the non-random weight
$\widetilde{\bm \Pi}_m^n$, i.e.,
\begin{equation}\label{eqproofth1.3} \sqrt{n}\sum_{m=1}^n \frac{\Delta\bm M_m}{m}\left(\bm \Pi_m^n-\widetilde{\bm \Pi}_m^n\right)\to \bm 0 \; \text{ in probability}.
\end{equation}
Now, $\big\{\frac{\Delta\bm M_m}{m} \widetilde{\bm \Pi}_m^n; m=1,\cdots, n\big\}$ is an array of martingale differences. By verifying the Lindeberg condition and checking the variance, we can show that
$$ \sqrt{n}\sum_{m=1}^n \frac{\Delta\bm M_m}{m}\widetilde{\bm \Pi}_m^n\overset{D}\to N(\bm 0,\bm\Sigma)\;(\text{stably}) $$
via the CLT for martingales. The above convergence and (\ref{eqproofth1.3}) are stated in Proposition \ref{prop2.1} in the Appendix. Thus, (\ref{eqTh1.1}) is proved.
The proof is now complete. \hfill $\Box$

 \bigskip

{\sc Proof of Theorem \ref{theorem3}.}
Recall (\ref{eqproofth1.1}) and $\bm H_n\to \bm H$ a.s. as $n\to\infty$. By (\ref{condition2}) and (\ref{eqproofth3.2}) we have
\begin{align*}
&\bm\theta_n-\bm\theta^{\ast}=(\bm\theta_0-\bm\theta^{\ast})\bm \Pi_0^n+\frac{\bm M_n+\bm s_n}{n}\bm \Pi_n^n+\sum_{m=1}^{n-1}(\bm M_m+\bm s_m)\frac{\bm I_d-\bm H_{m+1}}{m(m+1)}\bm \Pi_{m+1}^n\\
&= o(n^{-\rho+\delta})+\frac{o(n^{1/2+\delta/2})+o(n^{1-\rho})}{n} +\sum_{m=1}^{n-1}\frac{o(m^{1/2+\delta/2})+o(m^{1-\rho})}{m(m+1)}\big(\frac{n}{m}\big)^{-\rho+\delta}\\
& =o(n^{-\rho+\delta})\;\; a.s.
\end{align*}
It follows that
\begin{equation}\label{eqproof1.2} n^{\rho-\delta}(\bm\theta_n-\bm\theta^{\ast})\to \bm 0\;\; a.s. \text{ for all } \delta>0.
\end{equation}
According to (\ref{eqAssump1.1}), we can rewrite (\ref{eqModel}) as
$$ \bm\theta_{n+1}-\bm\theta^{\ast}=(\bm\theta_n-\bm\theta^{\ast})\left(\bm I_d-\frac{\bm H}{n+1}\right)+\frac{\bm r_{n+1}^{\ast}}{n+1}, $$
where
$ \bm r_{n+1}^{\ast}=o(\|\bm\theta_n-\bm\theta^{\ast}\|^{1+\epsilon})+\Delta \bm M_{n+1}+\bm r_{n+1}.$
From (\ref{condition2}), (\ref{eqproofth3.2}) and (\ref{eqproof1.2}), it follows that
$ \bm s_n^{\ast}=:\sum_{k=1}^n \bm r_k^{\ast}=o(n^{1-\rho-\delta})$ a.s.   for some $\delta>0$.
Recall that $\bm H$ has the Jordan canonical form
$  \bm T^{-1} \bm H\bm T= diag(\bm J_1, \cdots,\bm J_s), $
with  $\bm J_a=\lambda_a \bm I_{\nu_a}+\overline{\bm J}_{\nu_a}$. Denote
$(\bm \theta_n-\bm\theta^{\ast})\bm T:=\bm y_n=(\bm y_{n,1},\cdots,\bm y_{n,s})$,
$\bm r_n^{\ast}\bm T:=\widetilde{\bm r}_n=(\widetilde{\bm r}_{n,1},\cdots,\widetilde{\bm r}_{n,s})$, $\bm s_n^{\ast}\bm T:=\widetilde{\bm s}_n=(\widetilde{\bm s}_{n,1},\cdots,\widetilde{\bm s}_{n,s})$. Then,
$$ \bm y_{n+1,a}=\bm y_{n,a}\left(\bm I_{\nu_a} -\frac{\bm J_a}{n+1}\right)+\frac{\widetilde{\bm r}_{n+1,a}}{n+1}. $$
Write
$$ \widetilde{\bm \Pi}_k^{n,a} =\prod_{j=k+1}^n \left(\bm I_{\nu_a}-\frac{\bm J_a}{j}\right). $$
Then,
$$ \|\widetilde{\bm \Pi}_0^{n,a} \|\le C  n^{-Re(\lambda_a)}(\log n)^{\nu_a-1},\;\; \|\widetilde{\bm \Pi}_k^{n,a} \|\le  C \big(\frac{n}{k}\big)^{-Re(\lambda_a)}\big(\log \frac{n}{k}\big)^{\nu_a-1}, $$
$ 1\le k\le n$.
If $Re(\lambda_a)<1$ , then $ \widetilde{\bm \Pi}_0^{n,a} n^{\bm J_a}\to\bm A_a$,    $n^{-\bm J_a}(\widetilde{\bm \Pi}_0^{n,a})^{-1}\to \bm A_a^{-1}$
for an invertible matrix $\bm A_a$,   and
$ \|\widetilde{\bm \Pi}_0^{n,a} \|\approx   n^{-Re(\lambda_a)}(\log n)^{\nu_a-1}$. We have
\begin{align*}
\bm y_{n,a} =&\bm y_{0,0}\widetilde{\bm \Pi}_0^{n,a} +\sum_{k=1}^n \frac{\widetilde{\bm r}_{k,a}}{k}\widetilde{\bm \Pi}_k^{n,a}\\
=&\bm y_{0,0}\widetilde{\bm \Pi}_0^{n,a} +\frac{ \widetilde{\bm s}_{n,a}}{n}\widetilde{\bm \Pi}_n^{n,a}+\sum_{k=1}^{n-1} \frac{ \widetilde{\bm s}_{k,a}}{k}\frac{\bm I_{\nu_a}-\bm J_a}{k+1}\widetilde{\bm \Pi}_{k+1}^{n,a}.
\end{align*}
If $Re(\lambda_a)>\rho$, then
\begin{align*}\|\bm y_{n,a}\|=&O(1)n^{-Re(\lambda_a)}(\log n)^{\nu_a-1}+
o(n^{-\rho-\delta})\\
&+\sum_{k=1}^{n-1} \frac{ o(k^{-\rho-\delta})}{k+1}\big(\frac{n}{k}\big)^{-Re(\lambda_a)}\big(\log \frac{n}{k}\big)^{\nu_a-1}\\
=&o(n^{-\rho-\kappa})\;\; a.s.
\end{align*}
for some $0<\kappa<Re(\lambda_a)-\rho$. If $Re(\lambda_a)=\rho$, then
\begin{align*}\|\bm y_{n,a}\|=&O(1)n^{-\rho}(\log n)^{\nu_a-1}+
o(n^{-\rho-\delta})+\sum_{k=1}^{n-1} \frac{ o(k^{-\rho-\delta})}{k+1}\big(\frac{n}{k}\big)^{-\rho}\big(\log \frac{n}{k}\big)^{\nu_a-1}\\
=&O(1)n^{-\rho}(\log n)^{\nu_a-1}=o\big( n^{-\rho}(\log n)^{\nu-1} \big)\;\; a.s.\text{ when } \nu_a<\nu.
\end{align*}
Finally, consider the $\bm y_{n,a}$ with $Re(\lambda_a)=\rho$ and $\nu_a=\nu$. Note that
\begin{align*}
\bm y_{n,a}\big(\widetilde{\bm \Pi}_0^{n,a}\big)^{-1}
= \bm y_{0,a}  +\frac{ \widetilde{\bm s}_{n,a}}{n}\big(\widetilde{\bm \Pi}_0^{n,a}\big)^{-1}+\sum_{k=1}^{n-1} \frac{\widetilde{  \bm s }_{k,a}}{k}\frac{\bm I_{\nu_a}-\bm J_a}{k+1}\big(\widetilde{\bm \Pi}_0^{k+1,a}\big)^{-1}.
\end{align*}
Observe that
$$ \frac{ \widetilde{\bm s}_{n,a}}{n}\big(\widetilde{\bm \Pi}_0^{n,a}\big)^{-1}=o(n^{-\rho-\delta})n^{\rho}(\log n)^{\nu_a-1}\to 0\;\; a.s., $$
$$\sum_{k=1}^{\infty} \left\|\frac{\widetilde{  \bm s }_{k,a}}{k}\frac{\bm I_{\nu_a}-\bm J_a}{k+1}\big(\widetilde{\bm \Pi}_0^{k+1,a}\big)^{-1}
 \right\|\le c\sum_{k=1}^{\infty} \frac{O(k^{1-\rho-\delta})}{k^2}k^{\rho}(\log k)^{\nu_a-1}<\infty\;\;a.s. $$
 It follows that
 \begin{align*}
\bm y_{n,a}\big(\widetilde{\bm \Pi}_0^{n,a}\big)^{-1}
\to \bm y_{0,a} +\sum_{k=1}^{\infty} \frac{\widetilde{  \bm s }_{k,a}}{k}\frac{\bm I_{\nu_a}-\bm J_a}{k+1}\big(\widetilde{\bm \Pi}_0^{k+1,a}\big)^{-1} a.s.
\end{align*}
Thus,
 \begin{align*}
\bm y_{n,a}n^{\bm J_a}
\to \bm\xi_a=:\left[\bm y_{0,a} +\sum_{k=1}^{\infty} \frac{\widetilde{  \bm s }_{k,a}}{k}\frac{\bm I_{\nu_a}-\bm J_a}{k+1}\big(\widetilde{\bm \Pi}_0^{k+1,a}\big)^{-1}\right]\bm A_a\;\; a.s.
\end{align*}
It follows that
\begin{align*}
&\bm y_{n,a}= (\bm \xi_a+o(1))n^{-\bm J_a}
=(\bm \xi_a+o(1))n^{-\lambda_a}\exp\{-\overline{\bm J}_{\nu_a}\log n\} \\
=& \sum_{j=0}^{\nu_a-1}(\bm \xi_a+o(1))n^{-\lambda_a}\frac{(-\overline{\bm J}_{\nu_a})^j (\log n)^j}{j!} \\
=&  \xi_{a,\nu_a-1}  n^{-\lambda_a}(-1)^{\nu_a-1}\frac{(\log n)^{\nu_a-1}}{(\nu_a-1)!}(0,\cdots, 0,1)+ o\big( n^{-\rho} (\log n)^{\nu-1} \big)\\
=& \xi_{a,\nu_a-1} n^{-{\rm i}Im(\lambda_a)}(-1)^{\nu_a-1}\frac{n^{-\rho}(\log n)^{\nu_a-1}}{(\nu_a-1)!}(0,\cdots, 0,1)
+ o\big( n^{-\rho} (\log n)^{\nu-1} \big),
\end{align*}
because $(\overline{\bm J}_{\nu_a})^{\nu_a}=\bm 0$. Denote $\xi_a= \xi_{a,\nu_a-1}(-1)^{\nu_a-1}\frac{1}{(\nu_a-1)!}$.  Then,
\begin{align*}
& \frac{n^{\rho}}{(\log n)^{\nu-1}}\bm y_n-  \sum_{a: Re(\lambda_a)=\rho, \nu_a=\nu}  e^{-{\rm i}Im(\lambda_a)\log n} \xi_a\bm e_a\to \bm 0\;\; a.s.,\\
& \frac{n^{\rho}}{(\log n)^{\nu-1}}\big(\bm\theta_n-\bm\theta^{\ast}\big)- \sum_{a: Re(\lambda_a)=\rho, \nu_a=\nu}  e^{-{\rm i}Im(\lambda_a)\log n} \xi_a\bm e_a\bm T^{-1}\to \bm 0\;\; a.s.
 \end{align*}
 The proof is complete.\hfill $\Box$

\section{Gaussian process approximation}\label{GaussApp}

Write $\bm H=D\bm h(\bm\theta^{\ast})$. Suppose that $\bm B(t)$ is a $d$-dimensional standard Brownian motion that is independent of $\bm \Gamma$. Let $\bm G_t$ be a solution of the following differential equation
\begin{equation}\label{eqSDE}
d\,\bm G(t)=-\frac{\bm G(t)}{t}\bm H dt +\frac{ d\,\bm B(t)}{t}\bm\Gamma^{1/2},\;\; \bm G(1)=\bm G_1.
\end{equation}
It can be verified that
$$ \bm G(t)=\int_1^t \frac{ d \bm B(x)\bm \Gamma^{1/2}}{x} \Big(\frac{x}{t}\Big)^{\bm H}+\bm G_1 t^{-\bm H},\;\; t>0. $$
When $\bm G_1=\bm 0$, for a given $\bm \Gamma$, $\bm G(t)$ is a Gaussian process  with the  variance-covariance matrix
\begin{equation} \Var\{\bm G(t)\}=\frac{1}{t} \int_0^{\log t}  \big(e^{-( D\bm h(\bm \theta^{\ast})-\bm I/2)u})^{\rm t}\bm\Gamma e^{-(D\bm h(\bm \theta^{\ast})-\bm I/2)u}du.
\end{equation}
It is obvious that for a given $\bm \Gamma$, the limit variabilities in (\ref{eqTh1.2}) and (\ref{eqTh2.2}) are, respectively,
$$ \lim_{t\to \infty} t \Var\big\{  \bm G(t)\big\} \;\; \text{ and }\;\; \lim_{t\to \infty}\frac{t}{\log^{2\nu-1} t} \Var\big\{  \bm G(t)\big\}. $$
The next theorem shows that $\bm\theta_n-\bm\theta^{\ast}$ can be approximated by the Gaussian process $\bm G(t)$ under certain conditions. From the Gaussian approximation, we can obtain the law of the iterated logarithm for $\bm \theta_n-\bm\theta^{\ast}$ and the functional central limit theorem for the process $\bm \theta_{[nt]}-\bm\theta^{\ast}$.
\begin{theorem}\label{theoremGauss} Suppose that Assumption \ref{assump1} is satisfied, $\bm\theta_n\to\bm\theta^{\ast}$ a.s. and
\begin{align}
&\sum_{m=1}^n \bm r_m=o(n^{1/2-\epsilon_0})\; a.s.,\label{eqThG.1}\\
&\sum_{m=1}^{\infty} \ep\left[\|\Delta\bm M_m\|^2 I\{\|\Delta\bm M_m\|^2\ge m^{1-\epsilon_0}\}\big|\mathscr{F}_{m-1}\right]/m^{1-\epsilon_0}<\infty \;\; a.s., \;\text{and} \label{eqThG.2}\\
& \sum_{m=1}^n \ep\left[(\Delta\bm M_m)^{\rm t}\Delta\bm M_m\big|\mathscr{F}_{m-1}\right]=n\bm \Gamma+o(n^{1-\epsilon_0})\;\; a.s.\;\text{ or in }\; L_1
\label{eqThG.3}
\end{align}
for some $0<\epsilon_0<1$, where $\bm \Gamma$ is a symmetric positive semidefinite matrix that is  $\mathscr{F}_m$-measurable for some $m$.
Then, ({\em possibly in an enlarged
probability space with the process $\{(\bm\theta_n, \bm M_n,\bm r_n); n\ge 1\}$ being
redefined without changing its distribution}) there is  a $d-$dimensional standard Brownian
motions $\bm B(t)$ that is independent of $\bm\Gamma$, such that
\begin{equation}\label{eqThG.4}
\bm \theta_n-\bm \theta^{\ast}=\bm G(n)+o(n^{-1/2-\tau})\;\;a.s. \; \text{for some } \tau>0,
\end{equation}
when $\rho>1/2$, and
\begin{equation}\label{eqThG.5}
\bm \theta_n-\bm \theta^{\ast}=\bm G(n)+O\big(n^{-1/2}\log^{\nu-1}n\big)\;\; a.s.
\end{equation}
when $\rho=1/2$, where $\bm G(t)$ is the solution of equation (\ref{eqSDE}).
\end{theorem}

\begin{remark} The proof of Gaussian approximation is based on the strong approximation theorems for multivariate martingales. The condition that $\bm \Gamma$   is  $\mathscr{F}_m$-measurable for some $m$ is given by Eberlein (1986), Monrad and Philipp (1991) and Zhang (2004) to establish   the strong approximation theorems for multivariate martingales. For general random $\bm \Gamma$, the strong approximation is unknown. In practice, $\bm \Gamma$ is usually assumed to be deterministic.
\end{remark}

\begin{proof} Note conditions (\ref{eqThG.2}) and (\ref{eqThG.3}). By Theorem 1.3 of Zhang (2004), possibly in an enlarged
probability space with the process $\{(\bm\theta_n, \bm M_n,\bm r_n); n\ge 1\}$
redefined without changing its distribution, there is   $d-$dimensional standard Brownian
motions $\bm B(t)$ independent of $\bm\Gamma$, such that
\begin{equation}\label{eqproofThG.1}
\bm M_n=\bm B(n)\bm\Gamma^{1/2}+o(n^{1/2-\tau})\; a.s. \; \text{ for some }\; \tau>0.
\end{equation}
Let $\bm G(t)$ be the solution of equation (\ref{eqSDE}). By some elementary calculation we can write
\begin{align*}
\bm G(n+1)-\bm G(n)=& -\int_n^{n+1} \frac{\bm G(x)}{x} dx \bm H+\int_n^{n+1}\frac{ d \bm B(x)}{x} \bm \Gamma^{1/2}\\
=& -\frac{\bm G(n)}{n+1}\bm H+\frac{ [\bm B(n+1)-\bm B(n)]\bm\Gamma^{1/2}+\bm\delta_{n+1}}{n+1}
\end{align*}
with
\begin{equation}\label{eqproofThG.2} \sum_{m=1}^{\infty}\bm \delta_m \; \text{ being convergent } a.s.
\end{equation}
According to (\ref{eqModel}), we have
$$ \bm \theta_{n+1}-\bm\theta^{\ast}=(\bm\theta_n-\bm\theta^{\ast})\Big(\bm I_d-\frac{\bm H}{n+1}\Big)+
\frac{\Delta \bm M_{n+1}+[(\bm \theta_n-\bm\theta^{\ast})\bm H -\bm h(\bm\theta_n)]+\bm r_{n+1}}{n+1}. $$
It follows that the sequence $\{\bm\theta_n-\bm \theta^{\ast}-\bm G(n)\}$ satisfies
$$ \bm \theta_{n+1}-\bm\theta^{\ast}-\bm G(n+1)=(\bm\theta_n-\bm\theta^{\ast}-\bm G(n))\Big(\bm I_d-\frac{\bm H}{n+1}\Big)+
\frac{\bm r_{n+1}^{\ast}}{n+1} $$
with
$$ \bm r_{n+1}^{\ast}=\bm r_{n+1}-\bm\delta_{n+1}+\left[(\bm\theta_n-\bm\theta^{\ast})\bm H-\bm h(\bm\theta_n)\right]+\Delta\big(\bm M_{n+1}-\bm B(n+1)\bm\Gamma^{1/2}\big). $$
According to (\ref{eqproofth3.5}),
\begin{equation}\label{eqproofThG.3} (\bm\theta_n-\bm\theta^{\ast})\bm H-\bm h(\bm\theta_n)
= o(\|\bm\theta_n-\bm\theta^{\ast}\|^{1+\epsilon})=o(n^{-1/2-\epsilon/4}) \; a.s.
\end{equation}
 From (\ref{eqThG.1}), (\ref{eqproofThG.1}), (\ref{eqproofThG.2}) and (\ref{eqproofThG.3}), it follows that
$$ \bm s_n^{\ast}=:\sum_{m=1}^n \bm r_m^{\ast}=o(n^{1/2-\tau})\; a.s. \; \text{ for some }\; \tau>0. $$
   Recall that $\widetilde{\bm \Pi}_m^n=\sum_{j=m+1}^n \big(\bm I_d-\frac{\bm H}{j}\big)$ and $\|\widetilde{\bm \Pi}_m^n\|\le C_0\big(\frac{n}{m}\big)^{-\rho}\log^{\nu-1}\frac{n}{m}$  by Proposition \ref{lemma1}(i). Following the lines in (\ref{eqproofth1.1}) and (\ref{eqproofth1.2}), we conclude that
\begin{align*}   \bm\theta_n-&\bm\theta^{\ast} -\bm G(n)=(\bm\theta_0-\bm\theta^{\ast})\widetilde{\bm \Pi}_0^n+\frac{\bm s_n^{\ast}}{n}\widetilde{\bm \Pi}_n^n+\sum_{m=1}^{n-1}\bm s_m^{\ast}\frac{\bm I_d-\bm H}{m(m+1)}\widetilde{\bm \Pi}_{m+1}^n \\
=&O(1) n^{-\rho}\log^{\nu-1}n+\frac{o(n^{1/2-\tau})}{n} +\sum_{m=1}^{n-1}\frac{o(m^{1/2-\tau})}{m(m+1)}\big(\frac{n}{m}\big)^{-\rho}\log^{\nu-1}n \\
=&O\left(n^{-(1/2+\tau)\wedge \rho}\log^{\nu-1}n\right)\;\; a.s.
\end{align*}
The proof is complete.
\end{proof}

\section{Urn models}\label{sectionUrn}
\setcounter{equation}{0}
Urn models have long been considered
powerful mathematical instruments in many areas, including the physical
sciences, biological sciences, social sciences and engineering (Johnson and
Kotz, 1977; Kotz and Balakrishnan, 1997). The P\'olya urn
(also known as the P\'olya-Eggenberger urn) model was originally proposed to
model the problem of contagious diseases (Eggenberger and P\'olya, 1923). Since
then, there have been numerous generalizations and extensions. Among
them, the GFU (also known as
the generalized P\'olya urn or GPU in the literature) is the most popular (see  Athreya and Karlin, 1968; Athreya and Ney, 1972; Janson, 2004; etc.).
In clinical trial studies, response-adaptive designs for randomizing treatments
to patients aim at detecting "on-line" which
treatment should be assigned to more patients while retaining enough randomness to preserve the
basis of treatments. A large family of adaptive designs is based on the GFU (Wei and Durham, 1978; Wei, 1979; Smythe, 1996; Bai
and Hu, 1999, 2005;  Hu and Zhang, 2004; Hu and Rosenberger, 2006; Zhang, Hu and Cheung, 2006; Zhang et al., 2011; etc.). In this model, the adaptive approach relies on the cumulative information provided by the
responses to previous patients' treatments to adjust treatment allocation to the new
patients.
The idea
of this modeling is that the urn contains balls of $d$ different types representative of the treatments. At the beginning, the urn contains
$\bm Y_0=(Y_{0,1}, \cdots, Y_{0,d})\in \mathbb{R}^d\backslash \{\bm 0\}$ balls, where $Y_{0k}$
denotes the number of balls of type $k$, $k=1,\cdots, d$. At
stage $m$ ( $m=1,2,\cdots$), a ball is drawn from the urn with instant replacement. If the ball is of type $k$, then the $m$-th patient is allocated to treatment $k$, and additional $D_{k,q}(m)$ balls of type
$q$, $q=1,\cdots, d$, are added to the urn, where $D_{k,q}(m)$ may be a
function of another random variable $\bm \xi(m)$ and also may be a function of urn
compositions and the results of draws from previous stages. The random vector $\bm \xi(m)$ is usually the response of the $m$-th patient. This
procedure is repeated throughout $n$ stages. After $n$ draws and
generations, the urn composition is denoted by the row vector $\bm
Y_n=(Y_{n,1}, \cdots, Y_{n,d})$, where $Y_{n,k}$ is the number
of balls of type $k$ in the urn after the $n$th draw.
This
relation can be written as the following recursive formula:
\begin{equation}\label{eq1.1}
 \bm Y_n = \bm Y_{n-1}+\bm X_n \bm D_n, \end{equation}
 where $\bm D_n=\big(D_{k,q}(n)\big)_{k,q=1}^d $ and $\bm X_n$ is the result
of the $n$th draw, distributed according to the urn composition at
the previous stage, i.e., if the $n$th draw is a type $k$ ball,
then the $k$th component of $\bm X_n$ is $1$ and other components
are $0$. The matrices¡¯ $\bm D_n$s¡¯ are named as the adding
rules. The conditional expectations
$\bm H_n=\big(\ep[D_{k,q}(n)\big|\mathscr{F}_{n-1}]\big)_{k,q=1}^d $,
for given the history sigma field $\mathscr{F}_{n-1}$ generated by
the urn compositions  $\bm Y_1,\cdots, \bm Y_{n-1}$, the results of draws
$\bm X_1,\cdots, \bm X_{n-1}$ and   $\bm \xi(1),\cdots,
\bm \xi(n-1)$ of all previous stages, $n=1,2,\cdots$, are named as the
generating matrices.
 When $\bm
D_n$, $n=1,2,\cdots, $ are independent and identically distributed,
the GFU model is usually said to be homogeneous. In such a case,
$\bm H_n=\bm H$ are identical and nonrandom and the
adding rule $\bm D_n$ is merely a function of the $\bm \xi(n)$. In the general heterogeneous cases, both $\bm
D_n$ and $\bm H_n$ depend on the entire history of all of the stages.

Write $\bm N_n=(N_{n,1},\cdots, N_{n,d})$, where
$N_{n,k}$ is the number of times that a type $k$ ball is drawn in the
first $n$ stages. Also, in an adaptive design based on this urn model, $N_{n,k}$ is the number of patients being assigned to treatment $k$ after $n$ assignments.
Obviously,
\begin{equation}\label{eq1.2}
 \bm N_n=\sum_{k=1}^n \bm X_k. \end{equation}
    Athreya and Karlin (1967, 1968) first considered the asymptotic properties of
the homogeneous  GFU model  and conjectured that
$\bm N_n$ is asymptotically normal. Janson (2004) established the functional CLTs of $\bm Y_n$ and $\bm N_n$ for a  homogenous case in which the numbers of each type of balls were assumed to be integers.
   Bai and
Hu (2005) established   the asymptotic normality
for the  non-homogeneous GFU model under the following conditions:
\begin{align}
&\bm H_n\to \bm H\;\; a.s.,\;\; \bm H=(H_{k,j})_{d\times d},\; \; H_{k,j}\ge 0,   \label{CondBaiHu0} \\
&\sup\limits_{n\ge 1} \ep\left[\|\bm D_{n}\|^{2+\epsilon}\big|\mathscr{F}_{n-1}\right]<+\infty\;  a.s. \; \text{ for some } \epsilon>0, \label{CondBaiHu1} \\
&  \Cov\big[\{D_{q,k}(n), D_{q,l}(n)\}\big|\mathscr{F}_{n-1}\big]\to V_{qkl}\; a.s., \quad q,k,l=1,\cdots d,   \label{CondBaiHu2}
\\
&  \sum\limits_{m=1}^{\infty} \frac{\|\bm H_m-\bm H\|_{\infty}}{\sqrt{m}}<\infty \; a.s.,   \label{CondBaiHu3}\\
&  n\ep \|\bm H_n-\ep \bm H_n\|^2\to 0  \; a.s.,  \label{CondBaiHu3ad}\\
&   \bm H_n\bm 1^{\rm t}= \alpha\bm 1^{\rm t}\;\;\text{ with } \bm 1=(1,\cdots, 1)  \text{ for some } \alpha>0, \label{CondBaiHu4}
\end{align}
 and $\lambda_{sec}\le \alpha/2$,  where $\lambda_{sec}$ is the second largest real part of the eigenvalues of $\bm H$.    Higueras et al.  (2006) also considered the asymptotic normality of the urn compositions $\bm Y_n$ under the condition that
\begin{equation}\label{CondLarPages1}
n\ep\|\bm H_n-\bm H\|^2\to 0,
\end{equation}
which is weaker than  Bai and Hu's conditions  (\ref{CondBaiHu3}) and (\ref{CondBaiHu3ad}). However, Higueras et al.  (2006) only considered the case $\lambda_{sec}<\alpha/2$ and assumed an extra assumption; namely, that $\bm D_n\bm 1^{\prime}=\alpha\bm 1^{\prime}$, which is stricter than (\ref{CondBaiHu4}).

Laruelle  and Pag\`es (2013) derived  the joint asymptotic distribution of the vector $(\bm Y_n,\bm N_n)$  and weakened conditions  (\ref{CondBaiHu3}) and (\ref{CondBaiHu3ad}) to (\ref{CondLarPages1}). Moreover, the results only held when $\lambda_{sec}<\alpha/2$. In the study of adaptive designs driven by urn models,  $\lambda_{sec}\le  \alpha/2$ is a very stringent condition even when $d=3$ (cf. Chapter 4 of Hu and Rosenberger (2006)). The limit properties for $\lambda_{sec}>\alpha/2$ and for the case that (\ref{CondBaiHu4}) is not satisfied are stated as open problems in Hu and Rosenberger (2006, p. 158).
  In this section, we derive the joint asymptotic distribution of $(\bm Y_n, \bm N_n)$  by applying our new results on the SA algorithm (\ref{eqModel}). We consider  both the cases of $\lambda_{sec}\le \alpha/2$ and $\lambda_{sec}>\alpha/2$. We also remove condition (\ref{CondBaiHu4}) and weaken condition (\ref{CondBaiHu1}) to the  conditional Lindeberg condition.

Before we state the results, we first need some more notations and assumptions. To include various cases, we allow the numbers of balls to be non-integers and negative. For example, $D_{k,l}(n)<0$ means that $|D_{k,l}(n)|$ balls of type $l$ are removed from the urn when a ball of type of $k$ is drawn. We assume that a type of ball with a negative number  will never be selected and so,
$$ \pr(X_{n,k}=1|\mathscr{F}_{n-1})=\frac{ Y_{n-1,k}^+}{|\bm Y_{n-1}^+|}. $$
Here, $Y_{n,k}^+=\max\{Y_{n,k},0\}$ is the positive part of $Y_{n,k}$, $\bm Y_n^+=(Y_{n,1}^+,\cdots, Y_{n,d}^+)$, $|\bm Y_n^+|=\sum_{k=1}^d  Y_{n,k}^+$ and $\frac{0}{0}$ is defined as $\frac{1}{d}$, which means that each type of ball is selected with equal probability when the urn has no balls with a positive number. In this general framework, the urn allows negative and/or non-integer numbers of balls, removal and non-homogeneous updating. In considering the asymptotic properties, we need two assumptions on the adding rules.

\begin{assumption} \label{assumption4.1} Suppose that there is a  deterministic matrix $\bm H=(H_{q,k})_{q,k}^d$  with $H_{q,k}\ge 0$ for $q\ne k$
   such that
\begin{equation}\label{eqassumption4.1.2}
\sum_{m=1}^n \|\bm H_m-\bm H\|=o(n) \;\; a.s.
\end{equation}
 Further, assume that $\bm H$  has  a single  largest eigenvalue $\alpha>0$ and the
corresponding  left eigenvector $\bm v=(v_1,\cdots,v_d)$ and right
eigenvector $\bm u^{\rm t}=(u_1,\cdots, u_d)^{\rm t}$ such that $\sum_k
v_k=\sum_k v_k u_k=1$ and $ v_k>0, u_k>0$, $k=1,\cdots, d$.

 Without loss of generality, we assume that $\alpha=1$
throughout this paper. Otherwise, we may consider $\bm Y_m/\alpha$,
$\bm D_m/\alpha$ instead.
\end{assumption}
Assumption \ref{assumption4.1} means that  the updating is asymptotically stable and that on average, a draw will not generate the removal of the undrawn balls to avoid urn extinction, although balls of any type can be dropped from the urn at each specific stage.

When $\bm H$ satisfies the conditions in Assumption \ref{assumption4.1}, we let   $\lambda_2,\cdots,\lambda_s$ be the other eigenvalues of $\bm H$ and suppose  that $\bm H$ has the following Jordan canonical form
 decomposition
\begin{equation}\label{eq1.7}
 diag\left(
      1 , \bm
      J_2,\cdots,\bm J_s\right)
 \end{equation}
with $\bm J_t= \lambda_t \bm I_{\nu_t}+\overline{\bm J}_{\nu_t}$,
where $\nu_t$ is  the order of the Jordan block $\bm J_t$. Denote by $\lambda_{sec}=\max\{Re(\lambda_2),\cdots,Re(\lambda_s)\}$ and   $\nu=\max\{\nu_t:
Re(\lambda_t)=\lambda_{sec}\}$.

\begin{assumption} \label{assumption4.2}   Let
$ V_{qkl}(n)=:\Cov\big[\{D_{qk}(n), D_{ql}(n)\}\big|\mathscr{F}_{n-1}\big]$, $q,k,l=1,\cdots d $,
and denote by $\bm V_{nq}= (V_{qkl}(n))_{k,l=1}^d$.
Suppose that
\begin{equation}\label{eq4.2.1} \frac{1}{n}\sum_{m=1}^n \ep \left[\|\bm D_m\|^2I\{\|\bm D_m\| \ge \epsilon \sqrt{n}\} \big |\mathscr{F}_{m-1}\right] \to 0\;\; a.s.\; \text{ or in } L_1,\;\; \forall \epsilon>0,  \end{equation}
\begin{equation}\label{eq4.2.2} \frac{1}{n} \sum_{m=1}^n \bm V_{mq}\to \bm V_q\;a.s.\; \text{ or in } L_1\;\; \text{ for all }\;
q=1,\cdots,d,
\end{equation}
where $\bm V_q=(V_{qkl})_{k,l=1}^d$, $q=1,\cdots,d$, are  $d\times
d$ symmetric positive semidefinite   random matrices.
\end{assumption}

\subsection{Convergence Results}

An important step in showing the a.s. convergence of $\bm Y_n/n$ and $\bm N_n/n$ in Bai and Hu (1999, 2005) and Laruelle  and Pag\`es (2013) is  the convergence of $\frac{\sum_{k=1}^dY_{n,k}}{n}$, which is  derived from  the observation that
$$\left\{\sum_{k=1}^d(\Delta Y_{n,k}-\ep[\Delta Y_{n,k}|\mathscr{F}_{n-1}]); n\ge 1\right\}$$
is a martingale difference sequence where $\Delta Y_{n,k}=Y_{n,k}-Y_{n-1,k}$ and
$$ \sum_{k=1}^d\ep[\Delta Y_{n,k}|\mathscr{F}_{n-1}]=  \ep[\bm X_n\bm D_n\bm 1^t|\mathscr{F}_{n-1}]=\frac{\bm Y_{n-1}^+}{|\bm Y_{n-1}^+|}\bm H_n\bm 1^{\rm t}=\alpha. $$
The last equality above  is due to condition (\ref{CondBaiHu4}). When (\ref{CondBaiHu4}) is not satisfied, there is not an easy way to directly show the convergence of $\frac{\sum_{k=1}^dY_{n,k}}{n}$. Next, we modify the ODE method proposed by Laruelle  and Pag\`es (2013) to prove the convergence of $\bm Y_n/n$ and $\bm N_n/n$. The following theorem is the main result followed by its proof. Some of the basic tools in the ODE method that we used in the proof are presented in the Appendix.

\begin{theorem}\label{theorem4.1} Suppose that
 \begin{equation}\label{eqth4.1.0}   \sum_{m=1}^n \bm V_{mq}=O(n)\;a.s.\; \text{ or in } L_1\;\; \text{ for all }\;
q=1,\cdots,d,
\end{equation}
and Assumption \ref{assumption4.1} is  satisfied with $\lambda_{sec}<1$. Then,
\begin{equation}\label{eqth4.1} \frac{\bm Y_n}{n}\to \bm v\;a.s. \;\;\text{ and }\;\; \frac{\bm N_n}{n}\to \bm v.
\end{equation}
\end{theorem}

\begin{remark} It is easily seen that (\ref{eqth4.1.0}) is implied by either $\sup_m\ep\big[\|\bm D_m\|^2\big|\mathscr{F}_{m-1}\big]<\infty$ a.s.
or $\sup_m\ep \|\bm D_m\|^2<\infty$.
\end{remark}

\begin{proof} To prove this theorem, we note that $\bm Y_{n+1}=\bm Y_n+\bm X_{n+1}\bm D_{n+1}$ and $\ep[\bm X_{n+1}|\mathscr{F}_n]=\frac{\bm Y_n^+}{|\bm Y_n^+|}$.  Let $\Delta \bm M_{n,1}= \bm X_n-\ep[\bm X_n|\mathscr{F}_{n-1}]$ and  $\Delta \bm M_{n,2}=\bm X_n(\bm D_n-\ep[ \bm D_n|\mathscr{F}_{n-1}])$.  We have
\begin{align}\label{eqproofth4.1}
 \bm Y_{n+1}
 =&\bm Y_n+ \frac{\bm Y_n^+}{|\bm Y_n^+|} \bm H+\Delta \bm M_{n+1,1}\bm H +\Delta \bm M_{n+1,2}
  +\bm X_{n+1}(\bm H_{n+1}-\bm H).
  \end{align}
 Under assumption (\ref{eqassumption4.1.2}), we have $\sum_{m=1}^n\bm X_m(\bm H_m-\bm H)=o(n)$ a.s. It can be verified that    $\bm M_{n,1}\bm H +  \bm M_{n,2}=o(n)$ a.s. by assumption (\ref{eqth4.1.0}). Thus,
\begin{align*} \bm Y_n =&\bm Y_0+\sum_{m=0}^{n-1}\frac{\bm Y_m^+}{|\bm Y_m^+|} \bm H+(\bm M_{n,1}\bm H +  \bm M_{n,2})+\sum_{m=1}^n\bm X_m(\bm H_m-\bm H)\\
=&\sum_{m=0}^{n-1}\frac{\bm Y_m^+}{|\bm Y_m^+|} \bm H+o(n)\; a.s.
\end{align*}
It follows that
$$ \limsup_{n\to \infty} \frac{|Y_{n,k}|}{n}\le \limsup_{n\to \infty} \frac{1}{n}\sum_{m=0}^{n-1}\frac{|\sum_{q=1}^d Y_{m,q}^+H_{q,k}|}{|\bm Y_m^+|} \le  \max_{q,k}|H_{q,k}|\;\; a.s. $$
and
$$\liminf_{n\to \infty} \frac{\bm Y_n\bm u^{\rm t}}{n}=\liminf_{n\to \infty} \frac{1}{n}\sum_{m=0}^{n-1}\frac{\bm Y_m^+\bm u^{\rm t}}{|\bm Y_m^+|} \ge
\min_k u_k>0\;\; a.s. $$
Let $\Theta^{\infty}$ be the set of limiting values of $ \frac{\bm Y_n^+}{n}$ as $n\to \infty$. Then,
\begin{equation}\label{eqproofth4.3}\Theta^{\infty}\subset \left\{\bm\theta=(\theta_1,\cdots,\theta_d): \bm \theta\bm u^{\rm t}>0, 0\le \bm\theta_k\le \max_{q,l}|H_{q,l}|, k=1,\cdots,d\right\}.
\end{equation}
 Next, we show that
\begin{equation}\label{eqproofth4.4} \bm Y_n^+-\bm Y_n=o(n)\; a.s.
\end{equation}
Note that $|\bm Y_n^+|\ge c \bm Y_n^+\bm u^{\rm t}\ge c\bm Y_n\bm u^{\rm t}\to\infty$ a.s. as $n\to\infty$. Without loss of generality, we assume that $|\bm Y_n^+|>0$ for all $n$. Then, $X_{m+1,k}=0$ if $Y_{m,k}<0$. For $n$ and $k$, let $l_n=\max\{l\le n: Y_{l,k}\ge 0\}$ be the largest integer for which $Y_{l,k}\ge 0$. Then,
\begin{align*}
 Y_{n,k}=&Y_{l_n,k}+\sum_{m=l_n+1}^n \sum_{q=1}^dX_{m,q}D_{q,k}(m)\\
 =&Y_{l_n,k}+\sum_{m=l_n+1}^n \sum_{q=1}^dX_{m,q}[D_{q,k}(m)-H_{q,k}(m)]\\
 &+\sum_{m=l_n+1}^n \sum_{q=1}^dX_{m,q}[H_{q,k}(m)-H_{q,k}]
 + \sum_{m=l_n+1}^n \sum_{q=1}^dX_{m,q}H_{q,k} \\
 =&Y_{l_n,k} + \sum_{m=l_n+1}^n \sum_{q=1}^dX_{m,q}H_{q,k} +o(n)
 \ge  X_{l_n+1,k} H_{k,k}+o(n) \;\; a.s.
 \end{align*}
because $H_{q,k}\ge 0$ if $q\ne k$ and $X_{m,k}=0$ for $m=l_n+2,\cdots, n$. It follows that $\liminf_{n\to\infty}\frac{Y_{n,k}}{n}\ge 0$ a.s., and then (\ref{eqproofth4.4}) follows.

Now, write $\bm \theta_n=\frac{\bm Y_n^+}{n}$ and
\begin{align*}
\bm r_n = & (\Delta\bm M_{n,1}\bm H +  \Delta\bm M_{n,2})+ \bm X_n(\bm H_n-\bm H)+(\Delta\bm Y_n^+-\Delta\bm Y_n),\\
\bm s_n = &\sum_{m=1}^n \bm r_m= (\bm M_{n,1}\bm H +  \bm M_{n,2})+\sum_{m=1}^n\bm X_m(\bm H_m-\bm H)+(\bm Y_n^+-\bm Y_n)\\
&- \bm M_{0,1}\bm H -   \bm M_{0,2}-(\bm Y_0^+-\bm Y_0).
\end{align*}
 Then, from (\ref{eqproofth4.1}) and (\ref{eqproofth4.4}) we conclude that $\bm\theta_n$ is bounded with a probability of one and satisfies the SA algorithm (\ref{eqModel}) with
 $$ \bm h(\bm \theta)=\bm\theta\Big(\bm I_d-\frac{\bm H}{|\bm \theta|}\Big) \;  \text{ and }\; \frac{\bm s_n}{n}\to \bm 0 \;\; a.s., $$
where $|\bm \theta|=\sum_{k=1}^d |\theta_k|$. It is obvious that $\bm h(\bm \theta)$ is a continuous function on $\{\bm \theta:\bm\theta\bm u^{\rm t}>0\}$.

By Theorem \ref{thODE2} (a) and Remark \ref{remarkA.1}, the set $\Theta^{\infty}$ of the limiting values of $\bm\theta_n$   is a.s. a compact  connected set, stable by the
flow of the ordinary differential equation (ODE):
$$ \dot{\bm \theta}=-\bm h(\bm\theta). $$
It is obvious that $\bm h(\bm v)=0$. By Theorem \ref{thODE1}, $\Theta=:\{\bm \theta:\bm \theta\bm u^{\rm t}>0\}$ is a region of attraction of the above ODE for $\bm v$. Moreover, $\Theta$ is a neighborhood of $\bm v$. Further,  $\Theta^{\infty}\subset \Theta$ by (\ref{eqproofth4.3}). By Theorem \ref{thODE2} (b), we conclude that
  $$ \frac{\bm Y_n^+}{n}=\bm \theta_n  \to \bm v\;\; a.s.  $$
  Accordingly, $\bm Y_n/ n  \to \bm v$  a.s.,  $|\bm Y^+_n|/n \to |\bm v|=1$ a.s.

Finally,
\begin{align}\label{eqproofth4.5}
 \bm N_n=& \bm N_{n-1}+(\bm X_n-\ep[\bm X_n|\mathscr{F}_{n-1}])+\frac{\bm Y_{n-1}^+}{|\bm Y_{n-1}^+|}\nonumber\\
 &=\cdots =  \bm M_{n,1}- \bm M_{0,1}+\sum_{m=0}^{n-1}\frac{\bm Y_m^+}{|\bm Y_m^+|}.
\end{align}
It follows that
$$ \lim_{n\to \infty} \frac{\bm N_n}{n}=\lim_{n\to\infty} \frac{ \bm M_{n,1}}{n}+\lim_{n\to\infty} \frac{1}{n}\sum_{m=0}^{n-1}\frac{\bm Y_m^+/m}{|\bm Y_m^+|/m}=\bm v\;\; a.s. $$
The proof is complete.
\end{proof}

\subsection{Limiting Distribution}

We apply Theorems \ref{theorem2}-\ref{theorem1} to show  the rates of convergence. First,  we  show that the random vector $(\frac{\bm Y_n}{n},\frac{\bm N_n}{n})$ satisfies equation (\ref{eqModel}).  By (\ref{eqproofth4.1}) and (\ref{eqproofth4.5}) we have
\begin{align*}
 \bm Y_{n+1}
 =\bm Y_n+ \frac{\bm Y_n^+}{n} \frac{ \bm H}{|\bm Y_n^+/n|} +\Delta \bm M_{n+1,1}\bm H +\Delta \bm M_{n+1,2}
   +\bm X_{n+1}(\bm H_{n+1}-\bm H)
 \end{align*}
 and
 \begin{align}\label{eqforN}
 \bm N_{n+1}
  =&\bm N_n+\Delta \bm M_{n+1,1}+ \frac{\bm Y_n^+}{n} \frac{\bm I_d}{|\bm Y_n^+/n|} \nonumber\\
 =&\bm N_n+\Delta \bm M_{n+1,1}+\big(\frac{\bm Y_n^+}{n}-\bm v\big)\frac{n}{|\bm Y_n^+|}(\bm I_d-\bm 1^{\rm t}\bm v)+\bm v.
 \end{align}
Write
$\bm \theta_n=(\bm \theta_n^{(1)},\bm\theta_n^{(2)})=\big(\frac{\bm Y_n^+}{n}, \frac{\bm N_n}{n})$, $\bm\theta^{\ast}=(\bm v,\bm v)$,
$$\Delta\bm M_n= \big(\Delta \bm M_{n+1,1}\bm H+\Delta \bm M_{n+1,2},\Delta \bm M_{n+1,1}\big) $$
and
$$ \bm r_{n+1}=\Big(\bm X_{n+1} (\bm H_{n+1}-\bm H)+\Delta(\bm Y_{n+1}^+ -\bm Y_{n+1}), \bm 0\Big). $$
Then, $\bm \theta_n$ satisfies SA algorithm (\ref{eqModel}):
\begin{align}\label{eqproofEq}
 \bm \theta_{n+1}=  \bm \theta_n-\frac{\bm h(\bm\theta_n)}{n+1}+\frac{\Delta \bm M_{n+1}+ \bm r_{n+1}}{n+1},
\end{align}
 with
$$ \bm h(\bm\theta)=\bm \theta\begin{pmatrix} \bm I_d-\frac{\bm H}{|\bm \theta^{(1)}|} & -\frac{\bm I_d}{|\bm \theta^{(1)}|} \\\bm 0 & \bm I_d\end{pmatrix}. $$
For $ \bm r_{n+1}$, by noting that $Y_{n,q}$ is positive eventually and thus $Y_{n,q}=Y_{n,q}^+$  eventually due to   Theorem \ref{theorem4.1} and the fact that $v_q>0$, we have
\begin{equation}\label{eqY-Y}\bm Y_n^+ -\bm Y_n=O(1)\;\; a.s.\end{equation}
It follows that
\begin{equation}\label{eqRe}\sum_{m=1}^n\|\bm r_m\|=O(1)\sum_{m=1}^n\|\bm H_m-\bm H\|+O(1) \;\; a.s.
\end{equation}
For $\Delta \bm M_{n+1}$,   write $\bm\Sigma_{n,1}=  diag\big(\frac{\bm Y_n}{n}\big)-\frac{\bm Y_n^{\rm t}}{n}\frac{\bm Y_n}{n}$, $\bm \Sigma_{n,2}=\sum_{q=1}^d \frac{Y_{n,q}}{n} \bm V_{n+1,q}$. We have
$$ \ep[(\Delta\bm M_{n+1})^{\rm t}\Delta\bm M_{n+1}|\mathscr{F}_n]=\begin{pmatrix} \bm H^{\rm t} \bm\Sigma_{n,1}\bm H +\bm\Sigma_{n,2} & \bm H^{\rm t} \bm\Sigma_{n,1}\\
\bm\Sigma_{n,1}\bm H &   \bm\Sigma_{n,1} \end{pmatrix}. $$
Then, under Assumptions \ref{assumption4.1} and \ref{assumption4.2},
$$ \frac{1}{n}\sum_{m=1}^n  \ep[(\Delta\bm M_m)^{\rm t}\Delta\bm M_m|\mathscr{F}_{m-1}]\to \bm\Gamma \;\; a.s. \text{ or in } L_1, $$
where
$$\bm\Gamma=\begin{pmatrix} \bm H^{\rm t} \bm\Sigma_1\bm H +\bm\Sigma_2 & \bm H^{\rm t} \bm\Sigma_1\\
\bm\Sigma_1\bm H &   \bm\Sigma_1 \end{pmatrix},\;\; \bm\Sigma_1=diag(\bm v)-\bm v^{\rm t}\bm v,\;\; \bm\Sigma_2=\sum_{q=1}^d v_q\bm V_q.  $$
Finally, for $\bm h(\bm\theta)$,  it is easily seen that $\bm h(\bm\theta)$ is twice differentiable at $\bm\theta^{\ast}$ with
$$ D\bm h(\bm\theta^{\ast})=\begin{pmatrix} \bm I_d-(\bm H-\bm 1^{\rm t}\bm v) & -(\bm I_d-\bm 1^{\rm t}\bm v) \\ \bm 0 & \bm I_d\end{pmatrix}. $$
Obviously, the system of the eigenvalues of both $D\bm h(\bm\theta^{\ast})$ and $\bm I_d-(\bm H-\bm 1^{\rm t}\bm v)$ is $\{1, 1-\lambda_2,\cdots, 1-\lambda_t\}$. Thus,
$\rho=\min\{Re(1), Re(1-\lambda_2),\cdots, Re(1-\lambda_t)\}=1-\lambda_{sec}$. Further, it can be shown that if $\lambda_a\ne 0$,  then the largest order of  Jordan blocks  of both $D\bm h(\bm\theta^{\ast})$ and $\bm I_d-(\bm H-\bm 1^{\rm t}\bm v)$ with respect to their eigenvalue $1-\lambda_a$ is the same as the largest order of    Jordan blocks of $\bm H$ with respect to its eigenvalue $\lambda_a$.
Hence, by applying Theorems \ref{theorem1} and \ref{theorem2} we have the following central limit theorems for $(\bm Y_n,\bm N_n)$.

\begin{theorem} \label{theorem4.2}    Suppose that Assumptions \ref{assumption4.1} and \ref{assumption4.2} are satisfied.

 (i)  Further,   assume that $\lambda_{sec}<1/2$ and
 \begin{equation}\label{conditionTh4.2.1}
  \sum_{m=1}^n \bm \|\bm H_m-\bm H\|=o(\sqrt{n})\; \; a.s.\;  \text{ or in } L_1.
  \end{equation}
  Then
$$ \sqrt{n} \left(\frac{\bm Y_n}{n}-\bm v, \frac{\bm N_n}{n}-\bm v\right)   \overset{D}\to N(\bm 0,\widetilde{\bm \Sigma})\; (\text{stably}),$$
where
$$ \widetilde{\bm \Sigma}= \int_0^{\infty}\big(e^{-\bm Q u})^{\rm t}\bm\Gamma
e^{-\bm Q u}du $$
and
$$   \bm Q=\begin{pmatrix} \bm H-\bm 1^{\rm t}\bm v-\bm I_d/2 &  \bm I_d-\bm 1^{\rm t}\bm v \\ \bm 0 & -\bm I_d/2\end{pmatrix}. $$
 (ii)  Assume that $\lambda_{sec}=1/2$ and
 \begin{equation}\label{conditionTh4.2.2}
  \sum_{m=1}^n \frac{\bm \|\bm H_m-\bm H\|}{\sqrt{m}}=o(\sqrt{\log n}) \; a.s.\;\; \text{ or in } L_1.
  \end{equation}
  Then
$$ \frac{\sqrt{n}}{(\log n)^{\nu-1/2}} \left(\frac{\bm Y_n}{n}-\bm v, \frac{\bm N_n}{n}-\bm v\right)   \overset{D}\to N(\bm 0,\widetilde{\bm \Sigma})\; (\text{stably}),$$
where
$$ \widetilde{\bm \Sigma}= \lim_{n\to \infty} \frac{1}{(\log n)^{2\nu-1}}\int_0^{\log n}\big(e^{-\bm Q u})^{\rm t}\bm\Gamma
e^{-\bm Q u}du. $$
\end{theorem}

\begin{remark}
It can be verified that (\ref{CondBaiHu3}), which is  the condition of Bai and Hu (2005), implies (\ref{conditionTh4.2.1}) and (\ref{conditionTh4.2.2}). (\ref{conditionTh4.2.1}) is also weaker than (\ref{CondLarPages1}), which is condition (2.11) in Laruelle and Pag\`es (2013). Further, it can be verified that either (\ref{conditionTh4.2.1}) or  (\ref{conditionTh4.2.2}) implies
$$ \sum_{m=1}^n\|\bm H_m-\bm H\|=o(n^{1-\epsilon_0})\; a.s. \;  \text{ for some } \epsilon_0>0 $$
(cf. the proof of (\ref{eqproofth3.1})). Thus, condition (\ref{eqassumption4.1.2}) can be removed from the theorems.
\end{remark}


\begin{theorem} \label{theorem4.3}  Suppose that  Assumptions \ref{assumption4.1} and (\ref{eqth4.1.0}) are satisfied. Further, assume that $\lambda_{sec}>1/2$ and that
\begin{equation}\label{condition4.3.1} \sum_{m=1}^n \|\bm H_m-\bm H\|=o(n^{\lambda_{sec}-\delta_0})\;\; a.s. \;\text{ for some } \delta_0>0.
\end{equation}
Then, there are random complex  variables $\xi_2$, $\cdots$, $\xi_s$ and non-zero linearly independent left eigenvectors $\bm l_2$, $\cdots$, $\bm l_s$ of $\bm H$ with $\bm l_a\bm H =\lambda_a \bm l_a$ such that
\begin{equation}\label{eqTh4.3.2}\frac{n^{1-\lambda_{sec}}}{(\log n)^{\nu-1}}\big(\frac{\bm N_n}{n}-\bm v \big)- \sum_{a: Re(\lambda_a)=\lambda_{sec}, \nu_a=\nu}\hspace{-6mm} e^{{\rm i}Im(\lambda_t)\log n}  \xi_a\bm l_a(\bm I-\bm 1^{\prime}\bm v)\to \bm 0\; a.s.
\end{equation}
and
\begin{equation}\label{eqTh4.3.1} \frac{n^{1-\lambda_{sec}}}{(\log n)^{\nu-1}}\big(\frac{\bm Y_n}{n}-\bm v \big)-\frac{n^{1-\lambda_{sec}}}{(\log n)^{\nu-1}}\big(\frac{\bm N_n}{n}-\bm v \big)\bm H\to \bm 0\; a.s.
\end{equation}
\end{theorem}

\begin{proof} We apply Theorem \ref{theorem3} to prove this theorem.  Assume that $\bm T$ is a matrix such that
 $$\bm T^{-1}[\bm I_d-(\bm H-\bm 1^{\rm t}\bm v)]\bm T=diag\big(1, (1-\lambda_2)\bm I_{\nu_2}+\overline{\bm J}_{\nu_2},\cdots, (1-\lambda_s)\bm I_{\nu_s}+\overline{\bm J}_{\nu_s}\big). $$
 By (\ref{eqproofEq}), we have
$$ \frac{\bm Y_{n+1}^+}{n+1}  =  \frac{\bm Y_n^+}{n} -\frac{\bm h_1(\frac{\bm Y_n^+}{n})}{n+1}+\frac{\Delta \bm M_{n+1,1}\bm H+\Delta \bm M_{n+1,2}+\bm r^{(1)}_{n+1}}{n+1}, $$
where $\bm h_1(\bm\theta^{(1)})=\bm I_d-\frac{\bm H}{|\bm \theta^{(1)}|}$ is twice-differentiable with $D\bm h_1(\bm v)=\bm I_d-(\bm H-\bm 1^{\rm t}\bm v)$. Condition  (\ref{condition1}) is satisfied by assumption  (\ref{eqth4.1.0}),  and (\ref{condition2}) is  satisfied by (\ref{eqRe}) and assumption  (\ref{condition4.3.1}). Thus, by Theorem \ref{theorem3},
there are  complex random variables $\xi_2$, $\cdots$, $\xi_s$ such that
$$ \frac{n^{1-\lambda_{sec}}}{(\log n)^{\nu-1}}\big(\frac{\bm Y_n^+}{n}-\bm v \big)-\sum_{a: Re(\lambda_a)=\lambda_{sec}, \nu_a=\nu} e^{{\rm i}Im(\lambda_a)\log n} \xi_a\bm e_a\bm T^{-1}\to \bm 0\; a.s.
$$
From (\ref{eqforN}) and the above convergence,  we have
\begin{align*}
&\bm N_n-n\bm v\\
=& \bm M_{n,1}+
\sum_{m=1}^{n-1} \big(\frac{\bm Y_m^+}{m}-\bm v\big)\big(\frac{m}{\bm Y_m^+}-1\big)(\bm I_d-\bm 1^{\rm t}\bm v)+\sum_{m=1}^{n-1} \big(\frac{\bm Y_m^+}{m}-\bm v\big)(\bm I_d-\bm 1^{\rm t}\bm v)\\
=&O(\sqrt{n\log\log n})+\sum_{m=1}^{n-1}\big(O(1)\frac{(\log m)^{\nu-1}}{m^{1-\lambda_{sec}}}\big)^2\\
& +\sum_{m=1}^{n-1}\frac{(\log m)^{\nu-1}}{m^{1-\lambda_{sec}}}\sum_{a: Re(\lambda_a)=\lambda_{sec}, \nu_a=\nu} \hspace{-6mm} \big[ e^{{\rm i}Im(\lambda_a)\log m} \xi_a\bm e_a\bm T^{-1}+o(1)\big](\bm I_d-\bm 1^{\rm t}\bm v)
\\
=& n^{\lambda_{sec}}(\log n)^{\nu-1}\Big[o(1)
 +\sum_{a: Re(\lambda_a)=\lambda_{sec}, \nu_a=\nu}  \hspace{-6mm} e^{{\rm i}Im(\lambda_a)\log n} \lambda_a^{-1}\xi_a\bm e_a\bm T^{-1}(\bm I_d-\bm 1^{\rm t}\bm v)\Big]\; a.s.
\end{align*}
Write
$\bm l_a=\lambda_a^{-1}\bm e_a\bm T^{-1}(\bm I-\bm u^{\prime}\bm v)$ if $\lambda_a\ne 0$ and $\bm l_a= \bm e_a\bm T^{-1}(\bm I-\bm u^{\prime}\bm v)$ if $\lambda_a= 0$. Note that $\bm e_a\bm T^{-1}$ is a left eigenvector of $\bm H-\bm 1^{\prime}\bm v$ with respect to the eigenvalue $\lambda_a$.  We conclude that $\bm l_a\bm H=\lambda_a^{-1}\bm e_a\bm T^{-1}(\bm H-\bm u^{\prime}\bm v)=\lambda_a^{-1}\bm e_a\bm T^{-1}(\bm H-\bm 1^{\prime}\bm v)(\bm I-\bm u^{\prime}\bm v)=\lambda_a\bm l_a$ if $\lambda_a\ne 0$ and  $\bm l_a\bm H=\bm e_a\bm T^{-1}(\bm H-\bm 1^{\prime}\bm v)(\bm I-\bm u^{\prime}\bm v)=\bm 0$ if $\lambda_a=0$. Further,  $\bm l_a(\bm I-\bm 1^{\prime}\bm v)=\lambda_a^{-1}\xi_a\bm e_a\bm T^{-1}(\bm I_d-\bm 1^{\rm t}\bm v)$ and $\bm l_a(\bm I-\bm 1^{\prime}\bm v)\bm H=\lambda_a^{-1}\bm e_a\bm T^{-1}(\bm H-\bm 1^{\rm t}\bm v)=\bm e_a\bm T^{-1}$ when $\lambda_a\ne 0$.  (\ref{eqTh4.3.2}) is  proved, and
(\ref{eqTh4.3.1}) is also proved by noting (\ref{eqY-Y}). Finally, the linear independence of $\bm l_2$,$\cdots$, $\bm l_s$  is due to the linear independence of the system   $\{\bm v(=\bm e_1\bm T^{-1}), \bm e_2\bm T^{-1},\cdots,\bm e_s\bm T^{-1} \}$.
\end{proof}

\begin{remark} When $\lambda_{sec}>1/2$, Bai and Hu (2005) showed that $\bm Y_n-n\bm v=O\big(n^{\lambda_{sec}}\log^{\nu-1}n\big)$ in probability.
Now, by Theorem \ref{theorem4.3}, $\bm Y_n-n\bm v=O\big(n^{\lambda_{sec}}\log^{\nu-1}n\big)$ a.s. and $\bm N_n-n\bm v=O\big(n^{\lambda_{sec}}\log^{\nu-1}n\big)$ a.s. Further, if all eigenvalues with $Re(\lambda_t)=\lambda_{sec}$ and $\nu_t=\nu$ are real, then both
$(\bm Y_n-n\bm v)/(n^{\lambda_{sec}}\log^{\nu-1}n)$   and $(\bm N_n-n\bm v)/(n^{\lambda_{sec}}\log^{\nu-1}n)$  a.s. converge toward a finite random vector.
\end{remark}


\appendix


\setcounter{equation}{0}

\section{ODE methods for the recursive algorithm }
\setcounter{equation}{0}

\begin{theorem}\label{thODE2} (Kushner-Clark) Consider the following recursive procedure
$$ \bm\theta_{n+1}=\bm \theta_n-\gamma_{n+1}\bm h(\bm\theta_n)+\gamma_{n+1}\bm r_{n+1}, \eqno (RP) $$
where $\bm h$ is a continuous function and $\{\gamma_n\}$ is a positive sequence that tends toward zero, such that $\sum_{n=1}^{\infty}\gamma_n$ diverges.

(a) We suppose that sequence $\{\bm \theta_n\}$ is bounded, and for all $T>0$,
\begin{equation} \label{eqthODE2.1}\lim_{n\to \infty} \sum_{j\le m(n,T)}\left\|\sum_{k=n}^j \gamma_{k+1}\bm r_{k+1}\right\| =0,
\end{equation}
where $ m(n,T)=\inf\{k:k\ge n, \gamma_{n+1}+\cdots+\gamma_{k+1}\ge T\}. $
Then, set $\Theta^{\infty}$ of the limiting values of $\bm\theta_n$   is   a compact  connected set, made stable by the
flow of the ordinary differential equation:
$$ ODE_h\equiv\dot{\bm \theta}=-\bm h(\bm\theta). \eqno (ODE1)$$

(b) Further, let $\Theta$ be a region of attraction for $\bm \theta^{\ast}$, where $\bm\theta^{\ast}$ is a zero of $\bm h$,
i.e., the following properties are satisfied:
\begin{description}
  \item[\rm (i)] For any solution of (ODE1), if $\bm\theta(0)\in \Theta$, then $\bm\theta(s)\in \Theta$ for all $s\ge 0$;
  \item[\rm (ii)] if $\bm \theta$ is a solution of (ODE1) for which $\bm\theta(0)\in \Theta$, then
  $$ \bm\theta(s)\to  \bm\theta^{\ast}\;\; \text{ as } s\to +\infty; $$
  \item[\rm (iii)]  given $\epsilon>0$, there exists $\delta>0$ such that $\bm \theta(0)\in \Theta$ and $\|\bm \theta(0)-\bm\theta^{\ast}\|\le \delta$ imply
  $\|\bm \theta(s)-\bm\theta^{\ast}\|\le \epsilon$ for all $s\ge 0$.
\end{description}
 Suppose that
$\Theta$ is a neighborhood of $\bm\theta^{\ast}$. We assume the framework of part (a). If the sequence $\{\bm\theta_n\}$ returns infinitely often to a compact subset of $\Theta$, then it tends toward $\bm \theta^{\ast}$.
\end{theorem}

This is called the Kushner-Clark  theorem  and can be found in the book by Duflo (1997, p. 318).
A similar theorem  is obtained by Ljung (1977). Variants and improvements have been
proposed in classical textbooks by scholars such as  Duflo (1996, 1997), Kushner and Clark (1978) and Kushner and Yin (2003),  and in some papers (see e.g.,  Fort  and   Pag\`e, 1996).

\begin{remark} \label{remarkA.1} If  $\gamma_n\equiv\frac{1}{n}$ and $\frac{1}{n}\sum_{k=1}^n \bm r_k\to 0$, then (\ref{eqthODE2.1}) is satisfied.

In fact, let $\bm s_n=\sum_{k=1}^n \bm r_k$. Then, for $j\le m(n,T)$, we have $\sum_{k=n+1}^j \frac{1}{k}\le T$ and
\begin{align*}
\sum_{k=n}^j \frac{ \bm r_{k+1}}{k+1}=\sum_{k=n+1}^j \frac{\bm s_k}{k+1}\frac{1}{k}+\frac{\bm s_{j+1}}{j+1}-\frac{\bm s_n}{n+1}.
\end{align*}
It follows that
\begin{align*}  \max_{j\le m(n,T)}\left\|\sum_{k=n}^j \frac{\bm r_{k+1}}{k+1}\right\|\le & \sup_{m\ge n}\frac{\|\bm s_m\|}{m}\sum_{k=n+1}^{m(n,T)} \frac{1}{k}
+2\sup_{m\ge n}\frac{\|\bm s_m\|}{m}\\
 \le & (T+2)\sup_{m\ge n}\frac{\|\bm s_m\|}{m}\to 0\;\; a.s. \;\;\text{ as } n\to \infty.
 \end{align*}
 \end{remark}

\begin{theorem}\label{thODE1} Let $\bm H$ be a matrix satisfying Assumption 4.1. Suppose that $\bm u^{\rm t}>0$ and $\bm v>0$ are, respectively, the right and left eigenvectors of $\bm H$ with respect to the largest eigenvalue $1$ with $\bm v\bm 1^{\rm t}=1$ and $\bm v\bm u^{\rm t}=1$.
Consider the ordinary differential equation
$$ \dot{\bm\theta}=-\bm \theta\Big(\bm I_d-\frac{ \bm H} {|\bm \theta|}\Big), \;\; \bm\theta( 0)=\bm\theta_0, \eqno (ODE2) $$
where $|\bm \theta|=\sum_{k=1}^d |\theta_k|$. Then, $\Theta=\{\bm\theta:\bm\theta\bm u^{\rm t}>0\}$ is a region of attraction for $\bm v$.
\end{theorem}

\begin{proof} We need to verify   (i)-(iii) in Theorem \ref{thODE2}(b).  Suppose that $\bm\theta(0)=\bm \theta_0\in \Theta$.  By (ODE2), we have
\begin{equation}\label{eqproofODE1.1}\dot{\bm\theta}\bm u^{\rm t}=-\bm\theta\bm u^{\rm t} \Big(1-\frac{1}{|\bm \theta|}\big), and  \end{equation}
\begin{equation}\label{eqproofODE1.2} \dot{\bm\theta}(\bm I_d-\bm u^{\rm t}\bm v)=-\bm\theta(\bm I_d-\bm u^{\rm t}\bm v)
\Big(\bm I_d-\frac{\widetilde{\bm H}}{|\bm \theta|}\Big),
\end{equation}
where $\widetilde{\bm H}=\bm H-\bm u^{\rm t}\bm v$. Note that the eigenvalues of $\widetilde{\bm H}$ are $0,\lambda_2,\cdots,\lambda_t$.  Let
$$f(s)=\int_0^s\frac{1}{|\bm \theta(u)|} du. $$
Then, from (\ref{eqproofODE1.1}) and (\ref{eqproofODE1.2}), it follows that for all $s\ge 0$
\begin{equation}\label{eqproofODE1.3}\bm\theta(s)\bm u^{\rm t}=\bm\theta_0\bm u^{\rm t}\exp\{-(s-f(s))\}, \; \text{ and} \end{equation}
\begin{equation}\label{eqproofODE1.4} \bm\theta (s)(\bm I_d-\bm u^{\rm t}\bm v)=\bm\theta_0(\bm I_d-\bm u^{\rm t}\bm v)\exp\{-(s-f(s))\}\exp\{-f(s)(\bm I_d-\widetilde{\bm H})\}.
\end{equation}

Note that $|\dot{\theta_k}+\theta_k|\le \max_{i,j}|H_{i,j}|$ by (ODE2). It follows that $|\theta_k(s)|\le \max_{i,j}|H_{i,j}|+|\theta_{0,k}|$, which means that $\bm \theta(s)$ is bounded. Hence,
$$ 0<f(s)\to f(+\infty)=\int_0^{\infty} \frac{1}{|\bm\theta(u)|}du =+\infty\;\; \text{ as } \; s\to +\infty. $$
Note that $\bm \theta(s)\bm u^{\rm t}$ is bounded and $\bm \theta_0\bm u^{\rm t}>0$.  By (\ref{eqproofODE1.3}), it follows that
$0<f(s)<\infty$ for all $s>0$. Hence, $\bm \theta(s)\bm u^{\rm t}>0$ for all $s\ge 0$  by (\ref{eqproofODE1.3}) again. Thus, (i) is proved.

  By (\ref{eqproofODE1.3}) and (\ref{eqproofODE1.4}),
\begin{equation}\label{eqproofODE1.5} \big(\frac{\bm\theta (s)}{\bm\theta (s) \bm u^{\rm t}}-\bm v\big)=\big(\frac{\bm\theta_0}{\bm\theta_0 \bm u^{\rm t}}-\bm v\big)\exp\{-f(s)(\bm I_d-\widetilde{\bm H})\}\to \bm 0\;\; \text{ as }\; s\to \infty,
\end{equation}
because all eigenvalues $-1, \lambda_2-1,\cdots, \lambda_t-1$ of $-(\bm I_d-\widetilde{\bm H})$ have negative real parts.
Note that $\bm v>0$ and $\bm \theta(s)\bm u^{\rm t}>0$. By (\ref{eqproofODE1.5}), $\bm\theta(s)>0$ and so $|\bm \theta(s)|=\sum_{k=1}^d \theta_s(s)$ for an $s$ that is large enough. Hence, from (\ref{eqproofODE1.5}), it follows that
$\frac{|\bm\theta(s)|}{\bm \theta(s)\bm u^{\rm t}}\to 1$ as $s\to \infty$. Write $c(s)=\frac{\bm \theta(s)\bm u^{\rm t}}{|\bm\theta(s)|}- 1$. Then, by (\ref{eqproofODE1.1}), we have
$$ \dot{\bm\theta}\bm u^{\rm t}=-(\bm\theta\bm u^{\rm t}-1)+c(s)\;\; \text{ and } \; c(s)\to 0\;\; \text{ as } s\to +\infty. $$
It follows that
$$\bm\theta(s)\bm u^{\rm t}-1=(\bm\theta_0\bm u^{\rm t}-1)e^{-s}+e^{-s}\int_0^sc(u)e^{u}du\to 0\; \text{ as } s\to +\infty, $$
which, together with (\ref{eqproofODE1.5}), implies that
$$\bm \theta(s)=\bm v\bm u^{\rm t}\bm\theta(s)+o(1)\to \bm v\;\; \text{ as } s\to \infty. $$
Thus, (ii) is proved.

For (iii), we denote the solution of   (ODE2) with the initial value $\bm \theta_0\in \Theta$ by $\bm\theta(\bm \theta_0, s)$.  By (\ref{eqproofODE1.5}),
\begin{align*}
\frac{\bm\theta (\bm\theta_0, s)}{\bm\theta (\bm\theta_0,s) \bm u^{\rm t}}-\bm v =&\Big(\frac{\bm\theta_0}{\bm\theta_0 \bm u^{\rm t}}-\bm v\Big)\exp\big\{-f(s)(\bm I_d-\widetilde{\bm H})\big\} \\
=&(\bm \theta_0-\bm v)\Big(\bm I_d-\frac{\bm u^{\rm t}\bm \theta_0}{ \bm \theta_0\bm u^{\rm t}}\Big)\exp\big\{-f(s)(\bm I_d-\widetilde{\bm H})\big\}, \;\; s<\infty.
\end{align*}
Because all of the eigenvalues of $-(\bm I_d-\widetilde{\bm H})$ have negative real parts, there exists a constant $c_0>0$ such that
$ \big\|\exp\{-x(\bm I_d-\widetilde{\bm H})\}\big\|\le c_0$  for all  $ x\ge 0$.
It follows that
\begin{equation}\label{eqproofODE1.10}
\sup_{s\ge 0}\Big\|\frac{\bm\theta (\bm\theta_0, s)}{\bm\theta (\bm\theta_0,s) \bm u^{\rm t}}-\bm v \Big\|\le
C \|\bm\theta_0-\bm v\|\to 0\;\;\text{ as }\; \bm\theta_0\to \bm v.
\end{equation}
Write $c(\bm\theta_0,s)=\frac{\bm \theta(\bm\theta_0,s)\bm u^{\rm t}}{|\bm\theta(\bm\theta_0, s)|}- 1$. Then,
$c(\bm\theta_0,s)\to 0$ uniformly in $ s\ge 0$ as $ \bm\theta_0\to \bm v$
by (\ref{eqproofODE1.10}). Note that
$$  \bm\theta(\bm\theta_0,s)\bm u^{\rm t}-1=(\bm\theta_0\bm u^{\rm t}-1)e^{-s}+e^{-s}\int_0^sc(\bm\theta_0, u)e^{u}du. $$
It follows that
\begin{equation}\label{eqproofODE1.11}\sup_{s\ge 0}|\bm\theta(\bm\theta_0,s)\bm u^{\rm t}-1|\le |\bm\theta_0\bm u^{\rm t}-1|
+\sup_{s\ge 0}|c(\bm\theta_0,s)|.
\end{equation}
Hence,
\begin{align*}\bm\theta(\bm\theta_0,s)-\bm v=& \bm\theta (\bm\theta_0,s) \bm u^{\rm t}\Big(\frac{\bm\theta (\bm\theta_0, s)}{\bm\theta (\bm\theta_0,s) \bm u^{\rm t}}-\bm v \Big)+(\bm\theta (\bm\theta_0,s) \bm u^{\rm t}-1)\bm v\\
&\to 0\;   \text{ uniformly in } s\ge 0 \text{ as } \bm\theta_0\to \bm v
\end{align*}
by (\ref{eqproofODE1.10}) and (\ref{eqproofODE1.11}). Thus, (iii) is proved.
\end{proof}


\section{Basic results for  matrices and martingales}
 \setcounter{equation}{0}

\begin{proposition}\label{lemma1} Let $\{\bm H_n\}$ be a sequence of real matrices and $\bm H=D\bm h(\bm \theta^{\ast})$. Write $\bm \Pi_m^n=\prod_{j=m+1}^n \big(\bm I_d-\frac{\bm H_j}{j}\big)$  and $\widetilde{\bm \Pi}_m^n=\prod_{j=m+1}^n \big(\bm I_d-\frac{\bm H}{j}\big)$ for all $1\le m\le n-1$. Then,
\begin{description}
  \item[(i)] $\|\widetilde{\bm \Pi}_m^n\|\le C_0\big(\frac{n}{m}\big)^{-\rho}\log^{\nu-1}\frac{n}{m}\le C_{\delta}\big(\frac{n}{m}\big)^{-\rho+\delta}$ for all $\delta>0$;
  \item[(ii)] If $\bm H_n\to  \bm H$ as $n\to \infty$, then for all $\delta>0$,  $\|\bm \Pi_m^n \|\le  C_{\delta}\big(\frac{n}{m}\big)^{-\rho+\delta}$ and
      $$\bm \Pi_m^n-\widetilde{\bm \Pi}_m^n=o(1)\big(\frac{n}{m}\big)^{-\rho+\delta} \text{ as } n\ge m\to \infty; $$
  \item[(iii)]  If $\sum_{j=1}^{\infty}\frac{\|\bm H_j- \bm H\|}{j}(\log j)^{\nu-1}<\infty$, then   $\|\bm \Pi_m^n \|\le  C\big(\frac{n}{m}\big)^{-\rho}\log^{\nu-1}\frac{n}{m}$, and
      $$\bm \Pi_m^n-\widetilde{\bm \Pi}_m^n=o(1)\big(\frac{n}{m}\big)^{-\rho}\log^{\nu-1}\frac{n}{m} \text{ as } n\ge m\to \infty; $$
  \item[(iv)] $\max\limits_{x\in [m-c,m+c]}\|\widetilde{\bm \Pi}_m^n-\big(\frac{n}{x}\big)^{-\bm H}\|=o(1)\big(\frac{n}{m}\big)^{-\rho}\log^{\nu-1}\frac{n}{m}$ as $n\ge m\to \infty$.
\end{description}
Here, for a positive number $a$, $a^{\bm H}$ is defined as $a^{\bm H}=e^{\bm H\log a}=\sum_{j=0}^{\infty} \frac{1}{j!} (\log a)^j \bm H^j$.
\end{proposition}
\begin{proof} (i) can be found in Hu and Zhang (2004). For (ii), assume $\|\widetilde{\bm \Pi}_m^n\|\le C_{\delta}\big(\frac{n}{m}\big)^{-\rho+\delta}$. We show that there is a $m_{\delta}$ such that
\begin{equation}\label{eqproofLemm1.1} \|\bm \Pi_m^n \| \le 2C_{\delta}\big(\frac{n}{m}\big)^{-\rho+2\delta}
\end{equation}
  for all  $ m_{\delta}\le m\le n$.
Note that
\begin{align}\label{eqproofLemm1.2}
\bm \Pi_m^n=& \widetilde{\bm \Pi}_m^n+\sum_{k=m}^{n-1}\big(\bm \Pi_m^{k+1}  \widetilde{\bm \Pi}_{k+1}^n-\bm \Pi_m^k\widetilde{\bm \Pi}_{k}^n\big)
\nonumber\\
=& \widetilde{\bm \Pi}_m^n+\sum_{k=m}^{n-1}\bm \Pi_m^k\frac{\bm H-\bm H_{k+1}}{k+1}\widetilde{\bm \Pi}_{k+1}^n.
\end{align}
It follows that
$$ \|\bm \Pi_m^n\|\le C_{\delta} \big(\frac{n}{m}\big)^{-\rho+\delta}
+\sum_{k=m}^{n-1}\|\bm \Pi_m^k\|\frac{\|\bm H-\bm H_{k+1}\|}{k+1}C_{\delta}\big(\frac{n}{k+1}\big)^{-\rho+\delta}.
$$
We prove (\ref{eqproofLemm1.1}) by the induction.  We assume that (\ref{eqproofLemm1.1}) holds for all $ m_{\delta}\le n\le N-1$ and $ m_{\delta}\le m\le n$. Then, for $n=N$ and $ m_{\delta}\le m\le n$,
\begin{align*}
\|\bm \Pi_m^n\|\le &C_{\delta} \big(\frac{n}{m}\big)^{-\rho+\delta}+\sum_{k=m}^{n-1}2C_{\delta} \big(\frac{k}{m}\big)^{-\rho+2\delta}\frac{\|\bm H-\bm H_{k+1}\|}{k+1}C_{\delta}\big(\frac{n}{k}\big)^{-\rho+\delta}\big(\frac{k+1}{k}\big)^{\rho-\delta}
\\
\le &C_{\delta} \big(\frac{n}{m}\big)^{-\rho+\delta}+C_{\delta}^22^{1+\rho} \big(\frac{n}{m}\big)^{-\rho+2\delta}
\sum_{k=m_{\delta}}^{n}\frac{k^{\delta}\|\bm H-\bm H_{k+1}\|}{n^{\delta}(k+1)}.
\end{align*}
Note that
\begin{equation}\label{eqproofLemm1.3}
\sum_{k=m_{\delta}}^{n}\frac{k^{\delta}\|\bm H-\bm H_{k+1}\|}{n^{\delta}(k+1)}=n^{-\delta}\sum_{k=m_{\delta}}^{n}\frac{k^{\delta}o(1) }{k+1}\to 0 \;\; a.s. \text{ as } n\ge m_{\delta}\to \infty.
\end{equation}
We can choose an $m_{\delta}$ large enough that
 $$C_{\delta}^22^{1+\rho} \sum_{k=m_{\delta}}^{n}\frac{k^{\delta}\|\bm H-\bm H_{k+1}\|}{n^{\delta}(k+1)}<C_{\delta}\;\; a.s.$$
 Hence, we have
 \begin{align*}
\|\bm \Pi_m^n\|\le &C_{\delta} \big(\frac{n}{m}\big)^{-\rho+\delta}+ C_{\delta} \big(\frac{n}{m}\big)^{-\rho+2\delta}
\le  2 C_{\delta} \big(\frac{n}{m}\big)^{-\rho+2\delta}
\end{align*}
 for $m_{\delta}\le m\le n$. Thus, (\ref{eqproofLemm1.1}) is proved, and by (\ref{eqproofLemm1.2}) we have
\begin{align*}
\left\|\bm \Pi_m^n- \widetilde{\bm \Pi}_m^n\right\|\le & \sum_{k=m}^{n-1}2C_{\delta} \big(\frac{k}{m}\big)^{-\rho+2\delta}\frac{o(1)}{k+1}C_{\delta} \big(\frac{n}{k+1}\big)^{-\rho+\delta}\\
\le & \sum_{k=m}^{n-1}  \big(\frac{n}{m}\big)^{-\rho+2\delta}\frac{o(1)k^{\delta}}{n^{\delta}(k+1)} =o(1)\big(\frac{n}{m}\big)^{-\rho+2\delta}
\;\; a.s.
\end{align*}
as $n\ge m\to \infty$. The proof of (ii) is complete.

The proof of (iii) is similar if we note that
$$\big\|\sum_{k=m}^{n-1} \bm\Pi_m^k\frac{\bm H-\bm H_{k+1}}{k+1}\widetilde{\bm \Pi}_{k+1}^n\big\|
\le \sum_{k=m}^{n-1}\|\bm \Pi_m^k\|\frac{\|\bm H-\bm H_{k+1}\|}{k+1} C_0\big(\frac{n}{k+1}\big)^{-\rho}\log^{\nu -1}\frac{n}{k} $$
and
$$ \sum_{k=m}^{n-1}\frac{\|\bm H-\bm H_{k+1}\|}{k+1}(\log k)^{\nu-1} \to 0\;\;\text{as } n\ge m\to \infty. $$

For (iv), write
$  \big(\frac{j}{j-1}\big)^{\bm H}=\bm I_d-\frac{\bm H_j}{j}$, where $\bm H_j=\bm H+\frac{O(1)}{j}$.  From (iii), it follows that
$$ \widetilde{\bm\Pi}_m^n-\big(\frac{n}{m})^{-\bm H}=\widetilde{\bm \Pi}_m^n-\prod_{j=m+1}^n\big(\bm I_d-\frac{\bm H_j}{j}\big)=o(1)\big(\frac{n}{m}\big)^{-\rho}\log^{\nu-1}\frac{n}{m}. $$
The proof is completed by noting that $\max\limits_{x\in [m-c,m+c] }\big\|\big(\frac{m}{x}\big)^{-\bm H}-\bm I_d\|=O\big(\frac{1}{m}\big)$.
\end{proof}

\begin{proposition} \label{prop2.1} Suppose that Assumption 2.3 is satisfied, i.e.,
 \begin{equation} \label{eqAssump2.2} \frac{1}{n}\sum_{m=1}^n\ep\left[\|\Delta \bm M_m\|^2\mathbb I\{\|\Delta \bm M_m\|\ge \epsilon \sqrt{n}\}\big|\mathscr{F}_{m-1}\right]\to 0\;\; a.s.\; \text{ or in } L_1,\;\; \forall \epsilon>0
\end{equation}
\begin{equation} \label{eqAssump2.3}
\text{ and }\;\; \frac{1}{n}\sum_{m=1}^n \ep\left[(\Delta \bm M_m)^{\rm t} \Delta \bm M_m\big|\mathscr{F}_{m-1}\right]\to \bm \Gamma
\;\; a.s. \; \text{ or in } L_1,
\end{equation}
where $\bm \Gamma$ is a  symmetric positive semidefinite random matrix.

 Write $\bm H=D\bm h(\bm\theta^{\ast})$, $\bm \Pi_m^n=\prod_{j=m+1}^n \big(\bm I_d-\frac{\bm H_j}{j}\big)$  and $\widetilde{\bm \Pi}_m^n=\prod_{j=m+1}^n \big(\bm I_d-\frac{\bm H}{j}\big)$ and
$$\bm \zeta_n =\sum_{m=1}^n \frac{\Delta \bm M_m}{m}\widetilde{\bm \Pi}_m^n. $$
\begin{description}
 \item [(i)]
If $\rho>1/2$ and $\bm H_n\to \bm H$ a.s., then
\begin{equation}\label{eqprop2.1.1} \sqrt{n}\sum_{m=1}^n \frac{\Delta \bm M_m}{m} \bm \Pi_m^n-\bm \zeta_n\to \bm 0 \; \text{ in probability},
\end{equation}
$$ \sqrt{n}\bm\zeta_n \overset{D}\to N(\bm 0,\bm \Sigma)\;(\text{stably}),$$
where
$$ \bm \Sigma=\int_0^{\infty}\big(e^{-( \bm H-\bm I_d/2)u})^{\rm t}\bm\Gamma e^{-( \bm H-\bm I_d/2)u}du. $$
\item[(ii)]
If $\rho=1/2$, then
$$ \frac{\sqrt{n}}{(\log n)^{\nu-1/2}}\bm\zeta_n \overset{D}\to N(\bm 0,\widetilde{\bm \Sigma})\;(\text{stably}),$$
where
$$ \widetilde{\bm \Sigma}=\lim_{n\to \infty} \frac{1}{(\log n)^{2\nu-1}}\int_0^{\log n}\big(e^{-( \bm H-\bm I_d/2)u})^{\rm t}\bm\Gamma e^{-( \bm H-\bm I_d/2)u}du $$
satisfies (2.9), i.e.,
\begin{equation}\label{Alimitvaraince}
 ({\bm T^{\star}}^{\rm t}\widetilde{\bm \Sigma} \bm T)_{ij} =
\frac 1{((\nu-1)!)^2} \frac 1{2\nu-1}
     \bm t_{a 1}^{\star}\bm \Gamma\bm t_{b 1}^{\rm t},  \end{equation}
whenever $i=\nu_1+ \cdots+\nu_a $, $j=\nu_1+\cdots+\nu_{b}$ and $\lambda_a=\lambda_{b}$, $Re(\lambda_a)=1/2$,
$\nu_a=\nu_b=\nu$ and $ ({\bm T^{\star}}^{\rm t}\widetilde{\bm \Sigma} \bm
T)_{ij}=0$   otherwise. Here, $ \bm x^{\star} $ is the
conjugate vector of a complex   vector $\bm x$ and $\bm t_{a1}^{\rm t}$ is the first column vector of the $a$-th block in $\bm T=[ \cdots,  \bm t_{a1}^{\rm t},\cdots,\bm t_{a\nu_a}^{\rm t},\cdots]$.
 Further, let $\bm r_{a\nu_a}$ be the last row vector of the $a$-th block in $\bm T^{-1}=[\cdots,\bm r_{a1}^{\rm t},\cdots, \bm r_{a\nu_a}^{\rm t},\cdots]^{\rm t}$. Then, $\bm r_{a\nu_a}$ and $\bm t_{a1}^{\rm t}$ are respectively the left and right eigenvectors of $\bm H$ with respect to the eigenvalue $\lambda_a$, and
 $$  \widetilde{\bm \Sigma}  =
\frac 1{((\nu-1)!)^2} \frac 1{2\nu-1} \sum_{a,b: \lambda_a=\lambda_{b}, Re(\lambda_a)=1/2,
\nu_a=\nu_b=\nu}
     (\bm r_{a\nu_a}^{\rm t}\bm t_{a 1})^{\star}\bm \Gamma(\bm t_{b 1}^{\rm t} \bm r_{b\nu_n}). $$
\end{description}
\end{proposition}

\begin{proof} Without loss of generality, we assume that (\ref{eqAssump2.2}) and (\ref{eqAssump2.3}) hold in $L_1$. We have that
\begin{align*}\sum_{m=1}^n \frac{\Delta\bm M_m}{m}\bm \Pi_m^n= &\frac{\bm M_n}{n}\bm \Pi_n^n+\sum_{m=1}^{n-1}\bm M_m\frac{\bm I_d-\bm H_{m+1}}{m(m+1)}\bm \Pi_{m+1}^n,\\
\sum_{m=1}^n \frac{\Delta\bm M_m}{m}\widetilde{\bm \Pi}_m^n= &\frac{\bm M_n}{n}\widetilde{\bm \Pi}_n^n+\sum_{m=1}^{n-1}\bm M_m\frac{\bm I_d-\bm H}{m(m+1)}\widetilde{\bm \Pi}_{m+1}^n.
\end{align*}
When $\rho>1/2$, we choose $\delta>0$ such that $\rho-\delta>1/2$.  By Proposition \ref{lemma1} (ii) it follows that
\begin{align*}
&\Big\|\sum_{m=1}^n \frac{\Delta\bm M_m}{m}\big(\bm \Pi_m^n-\widetilde{\bm \Pi}_m^n\big)\Big\| \\
=& \Big\|\sum_{m=1}^{n-1} \frac{ \bm M_m}{m(m+1)}\big[(\bm I_d-\bm H)\big(\bm \Pi_{m+1}^n-\widetilde{\bm \Pi}_{m+1}^n\big)-(\bm H_{m+1}-\bm H)\bm \Pi_{m+1}^n\big]\Big\|\\
\le & \sum_{m=1}^{n-1} \frac{ \|\bm M_m\|}{m^2}o_{a.s.}(1) \big(\frac{n}{m}\big)^{-\rho+\delta}
=o_{a.s.}(1)\sum_{m=1}^{n-1} \frac{ O_{L_1}(m^{1/2})|}{m^2}  \big(\frac{n}{m}\big)^{-\rho+\delta}\\
=&o_P(1)n^{-\rho+\delta}\sum_{m=1}^{n-1}m^{-(3/2-\rho+\delta)}=o_P(n^{-1/2}).
\end{align*}
Thus, (\ref{eqprop2.1.1}) is  proved.

 Now, write
$ \bm \zeta_{n,m}=  \frac{\Delta \bm M_m}{m}\widetilde{\bm \Pi}_m^n, $
and $b_n=\sqrt{n}$ if $\rho>1/2$, $b_n=\sqrt{n}/(\log n)^{\nu-1/2}$ if $\rho=1/2$.
Then, $\{\bm\zeta_{n,m}; m=1,\cdots, n\}$ is an array of martingale differences. By  Corollary 3.1 of Hall and Heyde (1980, p. 58), it is sufficient to show that
\begin{align}\label{eqproofprop1.2}
b_n^2\sum_{m=1}^n \ep\left[  \|\bm \zeta_{n,m}\|^2\mathbb I\{b_n\|\bm \zeta_{n,m}\|\ge \epsilon\}\big|\mathscr{F}_{m-1}\right]\overset{P}\to 0, \; \text{ and} \\
 \label{eqproofprop1.3}
b_n^2\sum_{m=1}^n \ep\left[(\bm \zeta_{n,m})^{\rm t}\bm \zeta_{n,m} \big|\mathscr{F}_{m-1} \right]\overset{P}\to \bm\Sigma \;\; (resp. \; \widetilde{\bm \Sigma}).
\end{align}
We first verify (\ref{eqproofprop1.3}).  Write
$$ \bm S_n =\sum_{m=1}^n \ep\left[(\Delta \bm M_m)^{\rm t} \Delta \bm M_m\big|\mathscr{F}_{m-1}\right]-n \bm \Gamma . $$
Then,
\begin{align}\label{eqproofprop1.4}
&\sum_{m=1}^n \ep\left[(\bm \zeta_{n,m})^{\rm t}\bm \zeta_{n,m} \big|\mathscr{F}_{m-1} \right]
-\sum_{m=1}^n  (\widetilde{\bm \Pi}_m^n)^{\rm t} \frac{\bm \Gamma}{m^2} \widetilde{\bm \Pi}_m^n\nonumber\\
=&  \sum_{m=1}^n (\widetilde{\bm \Pi}_m^n)^{\rm t} \frac{\bm S_m-\bm S_{m-1} }{m^2} \widetilde{\bm \Pi}_m^n\nonumber\\
=  &  \frac{\bm S_n}{n^2}+ \sum_{m=1}^{n-1} \left[(\widetilde{\bm \Pi}_m^n)^{\rm t} \frac{\bm S_m }{m^2} \widetilde{\bm \Pi}_m^n- (\widetilde{\bm \Pi}_{m+1}^n)^{\rm t} \frac{\bm S_m }{(m+1)^2} \widetilde{\bm \Pi}_{m+1}^n\right]\nonumber\\
=&    \frac{\bm S_n}{n^2}+ \sum_{m=1}^{n-1} \left[  (\widetilde{\bm \Pi}_{m+1}^n)^{\rm t} \bm S_m \big(\frac{1}{m^2}-\frac{1 }{(m+1)^2}\big) \widetilde{\bm \Pi}_{m+1}^n\right] \nonumber\\
&+ \sum_{m=1}^{n-1}  (\widetilde{\bm \Pi}_m^n)^{\rm t} \frac{\bm S_m }{m^2}\left[ \widetilde{\bm \Pi}_m^n-\widetilde{\bm \Pi}_{m+1}^n\right]
 + \sum_{m=1}^{n-1}  \left[ \widetilde{\bm \Pi}_m^n-\widetilde{\bm \Pi}_{m+1}^n\right]^{\rm t} \frac{\bm S_m }{m^2} \widetilde{\bm \Pi}_{m+1}^n.
\end{align}
  Note by Proposition \ref{lemma1}(i) that
 $\| \widetilde{\bm \Pi}_m^n\|\le C\big(\frac{n}{m}\big)^{-\rho}(\log n/m)^{\nu-1}$   and $\|\widetilde{\bm \Pi}_m^n-\widetilde{\bm \Pi}_{m+1}^n\|=\|\frac{\bm H}{m+1} \widetilde{\bm\Pi}_{m+1}^n\|\le C\frac{1}{m}\big(\frac{n}{m}\big)^{-\rho}(\log n/m)^{\nu-1}$.  Each of the  terms in (\ref{eqproofprop1.4}) does not exceed
\begin{align*}
&\frac{\|\bm S_n\|}{n^2}+C  \sum_{m=1}^n  \frac{\|\bm S_m\|}{m^2}\frac{1}{m} \big(\frac{n}{m}\big)^{-2\rho}(\log \frac{n}{m})^{2(\nu-1)}\\
\overset{L_1}=&o(n^{-1})+  \sum_{m=1}^n  \frac{o(m)}{m^2}\frac{1}{m}\big(\frac{n}{m}\big)^{-2\rho}(\log \frac{n}{m})^{2(\nu-1)}=o(b_n^{-2}).
\end{align*}
In contrast, by noting  Proposition \ref{lemma1}  (iv),
it follows that
\begin{align*}
 &b_n^2\sum_{m=1}^n (\widetilde{\bm \Pi}_m^n)^{\rm t} \frac{\bm \Gamma}{m^2} \widetilde{\bm \Pi}_m^n  \\
=& b_n^2\sum_{m=1}^{n-1}\int_m^{m+1}\left(\frac{n}{y}\right)^{-\bm H^{\rm t}} \frac{\bm \Gamma}{y^2} \left(\frac{n}{y}\right)^{-\bm H }dy \\ &+b_n^2\sum_{m=1}^{n-1} o(1)\frac{1}{m^2}\big(\frac{n}{m}\big)^{-2\rho}\big(\log \frac{n}{m}\big)^{2(\nu-1)}\\
=& b_n^2 \int_1^n\left(\frac{n}{y}\right)^{-\bm H^{\rm t}} \frac{\bm \Gamma}{y^2} \left(\frac{n}{y}\right)^{-\bm H }dy +o(1)\\
=& \frac{b_n^2}{n} \int_0^{\log n} e^{-\bm H^{\rm t}u}\bm \Gamma e^{-\bm H u}e^u du+o(1)
= \bm \Sigma +o(1)\;\; \big(resp. \widetilde{\bm \Sigma} +o(1)\big).
\end{align*}
Thus, (\ref{eqproofprop1.3}) is proved.

To verify (\ref{eqproofprop1.2}), we first note that  (\ref{eqAssump2.2}) is equivalent to
\begin{equation} \label{eqproofprop1.5} \frac{1}{n}\sum_{m=1}^n\ep\left[\|\Delta \bm M_m\|^2\mathbb I\{\|\Delta \bm M_m\|\ge \epsilon \sqrt{m}\}\big|\mathscr{F}_{m-1}\right]\to 0\;\; a.s.\; \text{ or in } L_1\;\; \forall \epsilon>0.
\end{equation}
In contrast,
$$ b_n\|\bm\zeta_{n,m}\|\le C\frac{b_n}{\sqrt{n}} \frac{\|\Delta\bm M_m\|}{\sqrt{m}}\big(\frac{n}{m}\big)^{1/2-\rho}\big(\log\frac{n}{m}\big)^{\nu-1}
\le  C\frac{\|\Delta\bm M_m\|}{\sqrt{m}}.
$$
  Let  $S_n=\sum_{m=1}^n\ep\left[\|\Delta \bm M_m\|^2\mathbb I\{\|\Delta \bm M_m\|\ge \epsilon \sqrt{m}/C\}\big|\mathscr{F}_{m-1}\right]$ and\\
  $d_m=\frac{1}{\sqrt{m}}\big(\frac{n}{m}\big)^{1/2-\rho}\big(\log\frac{n}{m}\big)^{(\nu-1)}$. It follows that
\begin{align*}
&b_n^2\sum_{m=1}^n \ep\left[ \|\bm \zeta_{n,m}\|^2\mathbb I\{b_n\|\bm \zeta_{n,m}\|\ge \epsilon\}\big|\mathscr{F}_{m-1}\right]\\
\le & C\frac{b_n^2}{n}\sum_{m=1}^n d_m^2\ep\left[\|\Delta \bm M_m\|^2\mathbb I\{\|\Delta \bm M_m\|\ge \epsilon \sqrt{m}/C\}\big|\mathscr{F}_{m-1}\right]\\
=&C\frac{b_n^2}{n}\Big( \frac{S_n}{n}+\sum_{m=1}^{n-1}S_m(d_m^2-d_{m+1}^2)\Big)\le C\frac{b_n^2}{n}\Big( \frac{|S_n|}{n}+\sum_{m=1}^{n-1}\frac{|S_m|}{m}d_m^2 \Big)\\
= &  C\frac{b_n^2}{n}\Big( o(1)+\sum_{m=1}^{n-1}o(1)d_m^2 \Big)=o(1).
\end{align*}
Thus, (\ref{eqproofprop1.2}) is verified.

Finally, we verify (\ref{Alimitvaraince}) in the case of $\rho=1/2$. Note that
\begin{align*} &   \bm T^{-1} e^{-( \bm H-\bm I_d/2)u}\bm T=e^{-(\bm J-\bm I_d/2)u} \\
=& diag\big(e^{(1/2-\lambda_1)u}e^{-\overline{\bm J_1}u}, \cdots, e^{(1/2-\lambda_s)u}e^{-\overline{\bm J_s}u}\big),
\end{align*}
$$ {\bm T^{\star}}^{\rm t}\widetilde{\bm \Sigma}\bm T=\lim_{n\to \infty} \frac{1}{(\log n)^{2\nu-1}}\int_0^{\log n}{\big[e^{-(\bm J-\bm I_d/2)u}\big]^{\star}}^{\rm t}{\bm T^{\star}}^{\rm t}\bm\Gamma\bm T \big[e^{-(\bm J-\bm I_d/2)u}\big]du $$
and
\begin{align*} e^{-\overline{\bm J_a}u}=& \sum_{j=0}^{\nu_a-1}u^j \frac{(-\overline{\bm J_a})^j}{j!}=O(1)u^{\nu_a-2}+\frac{(-u)^{\nu_a-1}}{(\nu_a-1)!}(\overline{\bm J_a})^{\nu_a-1} \\
=& O(1)u^{\nu_a-2}+\frac{(-u)^{\nu_a-1}}{(\nu_a-1)!} (\underset{\nu_a}{\underbrace{1,0,\cdots, 0}})^{\rm t}(\underset{\nu_a}{\underbrace{0,\cdots,0, 1}})
\end{align*}
as $u\to +\infty$. Hence,
$$ e^{-(\bm J-\bm I_d/2)u}=O(1)u^{\nu-2}+(-1)^{\nu-1}\sum_{a:Re(\lambda_a)=\rho, \nu_a=\nu}e^{(1/2-\lambda_a)u}\frac{u^{\nu-1}}{(\nu-1)!} \bm f_a^{\rm t}\bm e_a   $$
as $u\to +\infty$, where  $\bm e_a=(\bm 0, \cdots,\bm 0,0,\cdots, 0,1,\bm 0,\cdots, \bm 0)$  is the vector such that   the $\nu_a$-th element of its block $a$  is $1$ and other elements are zero, and $\bm f_a=(\bm 0, \cdots,\bm 0,1,0,\cdots, 0,\bm 0,\cdots, \bm 0)$  is the vector such that   the first element of its block $a$  is $1$ and other elements are zero.
It is easily seen that
\begin{align*}
& \int_0^{\log n} \big(e^{(1/2-\lambda_a)u}u^j\big)^{\star}\big(e^{(1/2-\lambda_b)u}u^l\big)du=\int_0^{\log n} e^{(1-\lambda_a^{\star}-\lambda_b)u}u^{j+l} du \\
&\quad = \begin{cases} O(1), & \text{ if } Re(\lambda_a)+Re(\lambda_b)>1,\\
O((\log n)^{j+l}), & \text{ if } Re(\lambda_a)=Re(\lambda_b)=1/2, \lambda_a\ne \lambda_b,\\
\frac{(\log n)^{j+l+1}}{j+l+1}, & \text{ if } Re(\lambda_a)=Re(\lambda_b)=1/2, \lambda_a= \lambda_b.
\end{cases}
\end{align*}
It follows that
\begin{align*}
&\int_0^{\log n}{\big[e^{-(\bm J-\bm I_d/2)u}\big]^{\star}}^{\rm t}{\bm T^{\star}}^{\rm t}\bm\Gamma\bm T \big[e^{-(\bm J-\bm I_d/2)u}\big]du\\
=&\frac{1}{((\nu-1)!)^2}\frac{(\log n)^{2\nu-1}}{2\nu-1}\sum  (\bm f_a^{\rm t}\bm e_a)^{\rm t}{\bm T^{\star}}^{\rm t}\bm\Gamma\bm T (\bm f_b^{\rm t}\bm e_b)+O\big((\log n)^{2\nu -2}\big)\\
=&\frac{1}{((\nu-1)!)^2}\frac{(\log n)^{2\nu-1}}{2\nu-1}\sum (\bm f_a^{\rm t}{\bm T^{\star}}^{\rm t}\bm\Gamma\bm T \bm f_b^{\rm t})(\bm e_a^{\rm t}\bm e_b)+O\big((\log n)^{2\nu -2}\big)\\
=&\frac{1}{((\nu-1)!)^2}\frac{(\log n)^{2\nu-1}}{2\nu-1}\sum \bm t_{a1}^{\star}\bm\Gamma\bm t_{b1}^{\rm t} (\bm e_a^{\rm t}\bm e_b)+O\big((\log n)^{2\nu -2}\big),
\end{align*}
where the summation is taken over all $a,b$ with  $Re(\lambda_a)=1/2$, $\lambda_a=\lambda_b$ and $\nu_a=\nu_b=\nu$. So (\ref{Alimitvaraince}) is verified. The proof of (ii) is now complete.
\end{proof}

 \section{Some examples}\label{sectionexample}
\setcounter{equation}{0}
In this section, we give several examples  for the cases $\rho= 1/2$ and $\rho<1/2$.
The first example tells us that the elements of $\bm \theta_n$ may have different  convergence rates  if $\lambda_{\min}$ is a multiple eigenvalue and the order of a corresponding Jordan block of $D\bm h(\bm\theta^{\ast})$ exceeds one. The second shows that when $\lambda_{\min}$ is a complex eigenvalue it is possible that there is no $a_n$ for which $a_n(\bm \theta_n-\bm \theta^{\ast})$ has no zero limit. The third and  the last  show that the condition (\ref{eqAssump1.1}) cannot be weakened to (\ref{eqAssump0.1}) in Theorems \ref{theorem2} and \ref{theorem3}, and the convergence rates in conditions (\ref{conditionTh2.1}) or (\ref{conditionTh2.2}) cannot be weakened in Theorem \ref{theorem2}.

\begin{example}
Let $\theta_{0,1}=\theta_{0,2}=0$ and
\begin{align*}
\theta_{n+1,1}=& \theta_{n,1}-\frac{\lambda}{ n+1}\theta_{n,1}+\frac{\epsilon_{n+1}}{n+1},\\
\theta_{n+1,2}=& \theta_{n,2}-\frac{1}{(n+1)}(-\theta_{n,1}+\lambda\theta_{n,2}),
\end{align*}
where $\epsilon_n$ are i.i.d. standard normal random variables, $0<\lambda<1$.
That is
$$ \big(\theta_{n+1,1},\theta_{n+1,2}\big)=\big(\theta_{n,1},\theta_{n,2}\big)-
\frac{1}{n+1}\big(\theta_{n,1},\theta_{n,2}\big)\begin{pmatrix}  \lambda & -1 \\ 0 & \lambda\end{pmatrix} +\frac{1}{n+1}(\epsilon_{n+1},0).
$$
It is obvious that 
$$\bm h(\bm \theta)=\bm \theta \begin{pmatrix}  \lambda & -1 \\ 0 & \lambda\end{pmatrix}, $$
$\bm r_n=\bm 0$, $\Delta\bm M_n=(\epsilon_n,0)$ and $\rho=\lambda$.

However, if $\lambda=1/2$, then  $(\theta_{n,1},\theta_{n,2})\to (0,0)$ a.s.,
\begin{equation}\label{eqexample1} \sqrt{\frac{n}{\log n}} \theta_{n,1}\overset{D}\to N(0,1)\;\; \text{ and }\;\; \sqrt{\frac{3n}{(\log n)^2}} \theta_{n,2}\overset{D}\to N(0,1);
\end{equation}
and if $0<\lambda<1/2$, then there is a normal random variable $\xi\not\equiv 0$ such that
\begin{equation}\label{eqexample1.2} n^{\lambda} \theta_{n,1}\to \xi\; a.s. \; \text{ and }\;\;  \frac{n^{\lambda}}{\log n} \theta_{n,2}\to \xi\; a.s.
\end{equation}
\end{example}

\begin{proof} Let
$$ \Pi_k^n=\prod_{j=k+1}^n \Big(1-\frac{\lambda}{j}\Big), \;1 \le k\le n-1, \;\; \Pi_n^n=1.  $$
Then
$$ \theta_{n,1}=\left(1-\frac{\lambda}{n}\right)\theta_{n-1,1}+\frac{\epsilon_n}{n}=\cdots=\sum_{k=1}^n \Pi_k^n \frac{\epsilon_k}{k} $$
and
\begin{align*}
 \theta_{n,2}=&\sum_{k=1}^n \Pi_k^n \frac{\theta_{k-1,1}}{k}
 =  \sum_{k=1}^n \Pi_k^n \frac{1}{k}\sum_{j=1}^{k-1} \Pi_j^k \frac{\epsilon_j}{j}
 =  \sum_{j=1}^{n-1} \frac{\epsilon_j}{j} \Pi_j^n\sum_{k=j+1}^n\frac{1}{k-\lambda}
\end{align*}
are normal random variables with mean zeros.
Note that
$$ \sum_{k=j}^n\frac{1}{k}= \log\frac{n}{j}+o(1),\;\; \Pi_j^n\sim \big(\frac{j}{n}\big)^{\lambda}\;\; \text{ as } n\ge j\to \infty $$
and
$$ \sum_{k=j}^n\frac{1}{k}= \log\frac{n}{j}+O(1),\;\; \Pi_j^n\approx \big(\frac{j}{n}\big)^{\lambda}\;\; \text{ for all } n\ge j.$$
Suppose $\lambda=1/2$. Then
\begin{align*}
 Var(\theta_{n,1})=&\sum_{k=1}^n \frac{1}{k^2}(\Pi_k^n)^2 \sim \frac{1}{n} \sum_{k=1}^n \frac{1}{k}\sim \frac{\log n}{n},\\
 Var(\theta_{n,2})=&\sum_{j=1}^{n-1} \frac{1}{j^2}(\Pi_j^n\sum_{k=j+1}^n\frac{1}{k-\lambda})^2
\sim   \frac{1}{n} \sum_{j=1}^n \frac{1}{j} (\log \frac{n}{j})^2\\
\sim &\frac{1}{n} \int_1^n \frac{1}{x}(\log \frac{n}{x})^2dx =\frac{(\log n)^3}{3n}.
\end{align*}
(\ref{eqexample1}) is proved. Further, it is easily seen that
$$ \sum_n \pr(|\theta_{n,k}|\ge \epsilon)\le \frac{2}{\epsilon\sqrt{2\pi}}\sum_n \sqrt{Var(\theta_{n,k})}\exp\{ -\frac{\epsilon^2}{2 Var(\theta_{n,k})}\}<\infty, $$
since $\theta_{n,k}$ is a normal random variable with mean zero. So $\theta_{n,k}\to 0$ a.s. by the Borel-Cantelli lemma, $k=1,2$.

Now, suppose $0<\lambda<1/2$. Then
$$  \sum_{k=1}^{\infty} (\Pi_0^k)^{-1} \frac{\epsilon_k}{k} \;\; \text{ is a.s. convergent} $$
because
$$ \sum_{k=1}^{\infty} \Var\{(\Pi_0^k)^{-1} \frac{\epsilon_k}{k}\}\le C  \sum_{k=1}^{\infty}\frac{1}{k^{2-2\lambda}}<\infty. $$
Similarly,
$$ \sum_{k=1}^{\infty} (\Pi_0^j)^{-1} \frac{\epsilon_j}{j}\sum_{k=1}^{j-1}\frac{1}{k} \;\; \text{ is a.s. convergent}. $$
It follows that
$$\big(\Pi_0^n\big)^{-1}\theta_{n,1}=\sum_{k=1}^n \big(\Pi_0^k\big)^{-1} \frac{\epsilon_k}{k}\to \sum_{k=1}^{\infty} (\Pi_0^k)^{-1} \frac{\epsilon_k}{k}\;\; a.s.  $$
and
\begin{align*}\frac{1}{\log n} \big(\Pi_0^n\big)^{-1}\theta_{n,2}=& \frac{\sum_{k=1}^{n}\frac{1}{k-\lambda}}{\log n}\sum_{j=1}^{n-1} \frac{\epsilon_j}{j} \big(\Pi_0^j\big)^{-1}
-\frac{1}{\log n}\sum_{j=1}^{n-1} \frac{\epsilon_j}{j} \big(\Pi_0^j\big)^{-1}\sum_{k=1}^{j}\frac{1}{k-\lambda} \\
\to & \sum_{j=1}^{\infty} \frac{\epsilon_j}{j} \big(\Pi_0^j\big)^{-1}\;\; a.s.
\end{align*}
Note that $n^{\lambda}\Pi_0^n\to c_0$ and $\sum_{k=1}^{\infty} (\Pi_0^k)^{-1} \frac{\epsilon_k}{k}$ is a   normal random variable. (\ref{eqexample1.2}) is proved.
\end{proof}

\bigskip
The next gives an example   that there is no $a_n$ for which $a_n(\bm \theta_n-\bm \theta^{\ast})$ has no zero limit when $\rho<1/2$ and $\lambda_{\min}$ is a complex number.
\begin{example} Write $\bm \theta_n=(\theta_{n,1},\theta_{n,2})$, $\bm\epsilon_n=(\epsilon_{n,1},\epsilon_{n,2})$. Suppose that $\bm\theta_0=(0,0)$,
$$\bm \theta_{n+1}=\bm \theta_n -\frac{\lambda}{n+1}\bm \theta_n \begin{pmatrix}  1 & -1 \\ 1 & 1\end{pmatrix} +\frac{1}{n+1}\bm \epsilon_{n+1}, $$
where $\epsilon_{n,1}, \epsilon_{n,2}, n\ge 1$ are i.i.d. standard normal random variables, $\lambda>0$.

It is easily seen that $\bm r_n=\bm 0$, $\Delta\bm M_n=\bm\epsilon_n$,
$$ \bm h(\bm \theta)=\bm \theta \lambda \begin{pmatrix}  1 & -1 \\ 1 & 1\end{pmatrix}, \;\;
D\bm h(\bm\theta)=\lambda \begin{pmatrix}  1 & -1 \\ 1 & 1\end{pmatrix}. $$
The eigenvalues of $D\bm h(\theta)$ are $\lambda(1\pm i)$, and $\rho=\lambda$. However, when $\lambda<1/2$, 
  $\bm \theta_n\to \bm 0$ a.s., and there is no $a_n$ for which $a_n\bm \theta_n$ has no zero limit.
\end{example}
\begin{proof}  Let
$$ \bm T=\frac{1}{\sqrt{2}}\begin{pmatrix} 1 & i \\ 1 & -i \end{pmatrix}, $$
$ \bm y_n=\bm\theta_n\bm T^{-1}$, $ \bm\delta_n=\bm\epsilon_n\bm T^{-1}$. Then
$$ \bm T^{-1}= \frac{1}{\sqrt{2}}\begin{pmatrix} 1 & 1 \\ -i & i \end{pmatrix},\;\;
 \begin{pmatrix}  1 & -1 \\ 1 & 1\end{pmatrix}=\bm T^{-1}\begin{pmatrix}  1+i &  0 \\ 0 & 1-i\end{pmatrix} \bm T, $$

$$   y_{n+1,1}=  y_{n,1}-\frac{\lambda(1+i)}{n+1} y_{n,1}+
\frac{\delta_{n,1}}{n+1},\;\; \overline{y_{n,1}}=y_{n,2}. $$
Write
$$ \Pi_k^n=\prod_{j=k+1}^n \Big(1-\frac{\lambda(1+i)}{j}\Big). $$
Then $|\Pi_0^n|\approx n^{-\lambda}$, $n^{\lambda(1+i)} \Pi_0^n\to c_0\ne 0$,
$$ y_{n,1}=\sum_{k=1}^n \Pi_k^n \frac{\delta_{k,1}}{k}. $$
So
$$(\Pi_0^n)^{-1}y_{n,1}=\sum_{k=1}^n  \frac{\delta_{k,1}}{k\Pi_0^k}\to \sum_{k=1}^{\infty}  \frac{\delta_{k,1}}{k\Pi_0^k} \;a.s., $$
where the finiteness of the summation is due the fact that
$$ \sum_{k=1}^{\infty}  \ep\left|\frac{\delta_{k,1}}{k\Pi_0^k}\right|^2\le C\sum_{k=1}^{\infty} \frac{1}{n^{2(1-\lambda)}}<\infty. $$
 It follows that
$$ n^{\lambda(1+i)}y_{n,1}\overset{a.s.}\to \xi=:c_0\sum_{k=1}^{\infty}  \frac{\delta_{k,1}}{k\Pi_0^k}. $$
where $\xi=\xi_1+i\xi_2$ is a complex normal random variable with mean zero and
$$ \ep|\xi|^2=|c_0|\sum_{k=1}^{\infty}\frac{\ep|\delta_{1,1}|^2}{k^2|\Pi_0^k|^2}\ne 0. $$
  It follows that
\begin{align*}
&n^{\lambda}\frac{1}{\sqrt{2}}(\theta_{n,1}-i\theta_{n,2})=
 n^{\lambda} y_{n,1}
 =\xi n^{-\lambda i}+o(1)\\
 =& \left[\xi_1\cos(\lambda\log n)+\xi_2\sin(\lambda\log n)\right]
+i\left[\xi_2\cos(\lambda\log n)-\xi_1\sin(\lambda\log n)\right]+o(1)\;\; a.s.
\end{align*}
Hence
$$ n^{\lambda}\theta_{n,1}=\sqrt{2}\left[\xi_1\cos(\lambda\log n)+\xi_2\sin(\lambda\log n)\right]+o(1)\;\; a.s., $$
$$ n^{\lambda}\theta_{n,2}=\sqrt{2}\left[\xi_1\sin(\lambda\log n)-\xi_2\cos(\lambda\log n)\right]+o(1)\;\; a.s. $$
The proof is complete. 
\end{proof}

\bigskip
The next example  shows  the SA algorithm may have different rates of convergence though  the derivative $D\bm h(\bm\theta^{\ast})$ of the regression  function   $\bm h(\cdot)$ is  the same.

\begin{example} \label{example4} Let $\theta_0=e^{-e^2}$,
$$ \theta_{n+1}=\theta_n- \frac{\rho}{n+1}\int_0^{\theta_n}(1-f(x))dx, $$
where $0<\rho\le 1/2$, and $f$ is a continuous function on $[0,1]$ with $f(0)=0$. Then it is easily seen that
$h(\theta)=\rho\int_0^{\theta}(1-f(x))dx$, $D h(\theta)=\rho(1-f(\theta))$, $D h(0)=\rho$ and $\Delta M_{n+1}=r_{n+1}=0$. However,
\begin{description}
  \item[(i)] if $f(x)\equiv 0$, then
  $$ n^{\rho}\theta_n\to c_0>0; $$
    \item[(ii)] if $f(x)=\frac{1}{\log\log (1/(x\wedge \theta_0))}$, then
    $$n^{\rho-\epsilon}\theta_n\to 0\;\forall \; \epsilon>0\;\;\text{ and }\;\; \frac{n^{\rho}}{(\log n)^q}\theta_n \to +\infty\; \forall \; q>0. $$
\end{description}

\end{example}
\begin{proof} Let
$$ \Pi_k^n=\prod_{j=k+1}^n \left(1-\frac{\rho}{j}\right). $$
Then $\Pi_0^n\sim n^{-\rho} c_{\rho}$.  If $f(x)\equiv 0$, then
$$\theta_n =\Pi_0^n \theta_0\sim n^{-\rho}c_{\rho}\theta_0. $$
(i) is proved. Now suppose $f(x)=\frac{1}{\log\log (1/(x\wedge \theta_0))}$. Note $f(x)\ge 0$. We have
$$ \theta_n\ge \left(1-\frac{\rho}{n}\right)\theta_{n-1}\ge \cdots \ge \Pi_0^n \theta_0\ge c n^{-\rho}. $$
On the other hand,
$$\theta_n\le \left(1-\frac{\rho}{n}\frac{1}{2}\right)\theta_{n-1}\le\cdots\le \prod_{j=1}^n
\left(1-\frac{\rho}{2n}\right)\theta_0\le C n^{-\rho/2}, $$
because $0\le f(x)\le 1/2$. It follows that
$$ f(\theta_n)=\frac{1}{\log\log n}+\frac{O(1)}{(\log\log n)^2\log n}. $$
Write $\delta_{n+1} =\int_0^{\theta_n}f(x)dx/\theta_n$. Then $\theta_{n+1}=\left(1-\frac{\rho(1-\delta_{n+1})}{n+1}\right)\theta_n$ and
$$ \delta_{n+1}-f(\theta_n)\sim -\frac{1}{2}\frac{1}{(\log\log 1/\theta_n)^2\log 1/\theta_n}.$$
It follows that
$$\delta_{n+1}=\frac{1}{\log\log (n+1)}+\frac{O(1)}{(\log\log (n+1))^2\log (n+1)}. $$
Hence
\begin{align*}
 \theta_n=&\theta_0\prod_{j=1}^n\left(1-\frac{\rho(1-\delta_j)}{j}\right)\\
 =&\theta_0\exp\left\{-\sum_{j=1}^n\frac{\rho}{j}+\sum_{j=1}^n\frac{\rho}{j\log\log j}+c_n
 +\sum_{j=1}^n \frac{O(1)}{j(\log\log j)^2\log j} \right\} \\
 =& \theta_0\exp\left\{-\rho\log n+\frac{\rho\log n}{\log\log n}+C_n\right\},
 \end{align*}
where $c_n$ and $C_n$ are convergent sequences of real numbers. It follows that
$$ n^{\rho}\exp\left\{-\frac{\rho\log n}{\log\log n}\right\}\theta_n\to c>0. $$
 (ii) is proved. \end{proof}

\bigskip
The last example shows that  when $\rho=1/2$, the rate of convergence (\ref{eqLaruellePages1.2}) of the remainder $\bm r_n$ is not sufficient for investigating the asymptotic normality and  conditions (\ref{conditionTh2.1}) or (\ref{conditionTh2.2}) cannot be weakened.
\begin{example} \label{example5}
Let $\theta_{0}=0$ and
\begin{align*}
\theta_{n+1}=& \theta_{n}-\frac{1}{2(n+1)}\theta_{n}+\frac{\epsilon_{n+1}+r_{n+1}}{n+1},
\end{align*}
where $\epsilon_n$ are i.i.d. standard normal random variables, $r_n$ are real numbers.
Then
\begin{equation}\label{eqexample3.1}
 \theta_n-\ep\theta_n \to 0\;a.s.,\;\; \sqrt{\frac{n}{\log n}}(\theta_n-\ep \theta_n)\overset{D}\to N(0,1).
\end{equation}
However, if $r_n\equiv0$, then
$$  \ep \theta_n\equiv 0 \;\;\text{ and so }\;\;  \theta_n\to 0\; a.s.,\;\; \sqrt{\frac{n}{\log n}} \theta_n \overset{\mathscr{D}}\to N(0,1);$$
if
$r_{n}=\frac{1}{\sqrt{n}\log\log n}$, then
\begin{equation}\label{eqexample3.2}
  \ep \theta_n\sim \frac{\log n}{\sqrt{n}\log \log n}\;\; \text{ and so }\;\; \theta_n\to 0\; a.s.,\;\; \sqrt{\frac{n}{\log n}} \theta_n \overset{P}\to +\infty;
\end{equation}
if
$r_{n}=\frac{1}{\sqrt{n}\sqrt{\log n}}$, then
\begin{equation}\label{eqexample3.3}
  \ep \theta_n\sim 2\frac{\sqrt{\log n}}{\sqrt{n}}\;\; \text{ and so }\;\; \theta_n\to 0\; a.s.,\;\; \sqrt{\frac{n}{\log n}} \theta_n \overset{\mathscr{D}}\to N(2,1).
\end{equation}
\end{example}
Obviously, in each case $r_n$ satisfies (\ref{eqLaruellePages1.2}).

\begin{proof} Let
$$ \Pi_k^n=\prod_{j=k+1}^n \Big(1-\frac{1}{2j}\Big), \;\; \Pi_n^n=1.  $$
Then
$$\theta_n-\ep\theta_n=\sum_{k=1}^n\Pi_k^n \frac{\epsilon_k}{k},\;\;
\ep\theta_n=\sum_{k=1}^n\Pi_k^n \frac{r_k}{k}. $$
It follows that
$$ Var(\theta_n)=\sum_{k=1}^n(\Pi_k^n)^2 \frac{1}{k^2}\sim \sum_{k=1}^n \frac{k}{n}\frac{1}{k^2}\sim \frac{\log n}{n}.
$$
By noting that $\theta_n$ is a normal random variable, (\ref{eqexample3.1}) is proved.
It is obvious that $\ep \theta_n\equiv 0$ if $r_n\equiv 0$. If  $r_{n}=\frac{1}{\sqrt{n}\log\log n}$, then
$$ \ep \theta_n\sim \sum_{k=1}^n \big(\frac{k}{n}\big)^{1/2}\frac{1}{\sqrt{k}k\log\log k}
=\frac{1}{\sqrt{n}}\sum_{k=1}^n \frac{1}{k\log\log k}\sim \frac{\log n}{\sqrt{n}\log\log n}. $$
If  $r_{n}=\frac{1}{\sqrt{n}\sqrt{\log n}}$, then
$$ \ep \theta_n\sim \sum_{k=1}^n \big(\frac{k}{n}\big)^{1/2}\frac{1}{\sqrt{k}k\sqrt{\log k}}
=\frac{1}{\sqrt{n}}\sum_{k=1}^n \frac{1}{k\sqrt{\log k}}\sim 2\frac{\sqrt{\log n}}{\sqrt{n}}. $$
 (\ref{eqexample3.2}) is proved.
 \end{proof}



\begin{thebibliography}{9}

\bibitem{AK67}
{\sc Athreya, K. B.} and {\sc Karlin, S.} (1967). Limit theorems for the split
times of branching processes. {\em Journal of Mathematics and
Mechanics}  {\bf 17} 257-277. \MR{0216592}

\bibitem{AK68}
{\sc Athreya, K. B.} and {\sc Karlin, S.} (1968). Embedding of urn schemes into
continuous time branching processes and related limit theorems. {\em
Ann. Math. Statist.} {\bf 39} 1801-1817. \MR{0232455}

\bibitem{AN72}
{\sc Athreya, K. B.} and {\sc Ney, P. E.} (1972). {\em Branching Processes}. Springer, New York. \MR{0373040}

\bibitem{BH99}
{\sc Bai, Z. D.} and {\sc Hu, F.} (1999). Asymptotic theorem for urn models with
nonhomogeneous generating matrices. {\em Stochastic Process. Appl.}
{\bf 80} 87-101. \MR{1670107}

\bibitem{BH05}
{\sc Bai, Z. D.} and {\sc Hu, F.} (2005). Asymptotics in randomized urn models.
{\em Ann. Appl. Probab.} {\bf 15}  914-940. \MR{2114994}

\bibitem{BMP90}
 {\sc Benveniste, A.}, {\sc M\'etivier, M.} and {\sc Priouret, P.} (1990). {\em Adaptive algorithms and stochastic approximations}, volume
22 of {\em Applications of Mathematics (New York)}. Springer-Verlag, Berlin. Translated from the
French by Stephen S. Wilson. \MR{1082341}

\bibitem{B88}
{\sc Bouton, C.} (1988). Approximation Gausienne d'algorithms stochastiques a dynamique markovienne. {\em Ann. IHP}, serie {\rm I}. {\bf 24}: 131-155.

\bibitem{C10}
{\sc Cheung, Y. K.} (2010).  Stochastic approximation and modern model-based designs for dose-finding clinical
trials. {\em Statist. Sci.} {\bf 25}(2) 191-201. \MR{2789989}

\bibitem{D96} {\sc Duflo, M.} (1996). {\em  Algorithmes stochastiques}, volume 23 of {\em Math\'ematiques \& Applications (Berlin) [Mathematics
\& Applications]}. Springer-Verlag, Berlin. \MR{1612815}

\bibitem{D97}
{\sc Duflo, M.} (1997). {\em Random iterative models}, volume 34 of {\em Applications of Mathematics (New York)}. Springer-Verlag, Berlin. Translated from the 1990 French original by Stephen S. Wilson and revised by the
author. \MR{1485774}

\bibitem{Eb86} {\sc Eberlein, E.} (1986). On strong invariance principles under dependence assumptions. {\em Ann. Probab.} {\bf 14}: 260-270.
\MR{0815969}

\bibitem{ForPage96}
{\sc Fort, J.-C.} and {\sc Pag\`e, G.} (1996).
 Convergence of stochastic algorithms: from the Kushner-Clark
theorem to the Lyapunov functional method.
{\em Adv. Appl. Probab.} {\bf 28}(4):1072-1094.
\MR{1418247}

\bibitem{HH80} {\sc Hall, P.} and {\sc Heyde, C. C.} (1980). {\em Martingale Limit Theory and its Applications}. Academic Press, New York.
\MR{0624435}

\bibitem{HMPM03} {\sc Higueras, I.}, {\sc Moler, J.}, {\sc Plo, F.} and {\sc San Miguel, M.} (2003). Urn models and differential algebraic equation. {\em J. App. Probab.} {\bf 40} 401-412. \MR{1978099}

\bibitem{HMPM06} {\sc Higueras, I.}, {\sc Moler, J.}, {\sc Plo, F.} and {\sc San Miguel, M.} (2006). Central limit theorems for generalized P\'olya urn models.  {\em J. App. Probab.} {\bf 43}  438-451. \MR{2274628}


 \bibitem{HR06}
 {\sc Hu, F.}  and  {\sc Rosenberger, W. F.}  (2006).
{\em The Theory of Response-Adaptive Randomization in Clinical
Trials}. John Wiley and Sons, Inc., New York. \MR{2245329}

\bibitem{HRZ06}
{\sc Hu, F.}, {\sc Rosenberger, W. F.} and {\sc Zhang, L.-X.} (2006). Asymptotically best response-adaptive
randomization procedures. {\em J. Statist. Plann. Inference} {\bf 136}  1911-1922. \MR{2255603}

 \bibitem{HZ04a}  {\sc Hu, F.}  and   {\sc Zhang, L.-X.}   (2004). Asymptotic normality of urn
models for clinical trials with delayed response. {\em Bernoulli}
{\bf 10} 447-463. \MR{2061440}



 \bibitem{J04}
{\sc Janson, S.} (2004). Functional limit theorems for multitype branching
processes and generalized P\'olya urns. {\em Stochastic Process.
Appl.} {\bf 110} 177-245. \MR{2040966}

 \bibitem{JK77}
{\sc Johnson, N. L.} and {\sc Kotz, S.} (1977). {\em Urn Models and Their
Applications}. Wiley, New York. \MR{0488211}

\bibitem{LP13}
{\sc Laruelle, S.}  and {\sc Pag\`es, G.} (2013).
Randomized urn models revisited using stochastic approximation.
{\em Ann. Appl. Probab.} {\bf 23}(4)  1409-1436. \MR{3098437}

\bibitem{Ljung77}
{\sc Ljung, L.} (1977).
 Analysis of recursive stochastic algorithms. {\em IEEE Trans. Automatic Control} {\bf 22}(4):551-575.
\MR{0465458}

\bibitem{MP91} {\sc Monrad, D.} and  {\sc Philipp, W}. (1991). Nearby variables with nearby conditional laws and a strong approximation theorem
for Hilbert space valued martingales. {\em Probab. Theory Relat. Fields} {\bf 88}: 381¡§C404. \MR{1100898}

\bibitem{KC78}
 {\sc Kushner, H. J.} and  {\sc Clark, D. S.} (1978). {\em Stochastic approximation methods for constrained and unconstrained
systems}, volume 26 of {\em Applied Mathematical Sciences}. Springer-Verlag, New York. \MR{0499560}

\bibitem{KY03}
 {\sc Kushner, H. J.} and {\sc Yin, G. G.} (2003) {\em Stochastic approximation and recursive algorithms and applications},
volume 35 of {\em Applications of Mathematics (New York)}. Springer-Verlag, New York, second edition. {\em Stochastic Modelling and Applied Probability}. \MR{1993642}

\bibitem{KB97}  {\sc Kotz, S.} and {\sc Balakrishnan, N.} (1997).   Advances in urn models during the past two decades. In
{\em Advances in Combinatorial Methods and Applications to Probability and Statistics} (N. Balakrishnan,
ed.) 203-257. Birkh\"auser, Boston. \MR{1456736}

\bibitem{P98} {\sc Pelletier, M.} (1998). Weak convergence rates for stochastic approximation with application to multiple targets and simulated annealing. {\em Ann. Appl. Probab.}, {\bf 8} 10-34. \MR{1620405}

 \bibitem{Smythe96} {\sc Smythe, R. T.} (1996). Central limit theorems for urn models. {\em Stochastic Process. Appl.} {\bf 65} 115-137. \MR{1422883}

 \bibitem{W79} {\sc Wei, L. J.} (1979). The generalized P\'olya's urn design for sequential medical trials. {\it Ann. Statist.} {\bf 7} 291-296.

\bibitem{WD78}
{\sc Wei, L. J.} and {\sc Durham, S.} (1978). The randomized play-the-winner rule
in medical trials. {\it J. Amer. Statist. Assoc.} {\bf 73} 840-843.

\bibitem{Z04}
{\sc Zhang, L.-X.} (2004). Strong approximations of martingale vectors and its
applications in Markov-Chain adaptive designs. {\em Acta Math. Appl.
Sinica, English Series} {\bf 20}(2) 337-352.
 \MR{2064011}


\bibitem{ZHC06} {\sc Zhang, L.-X.}, {\sc Hu, F.} and {\sc Cheung, S. H.} (2006). Asymptotic theorems of sequential estimation-adjusted urn models for clinical trials.  {\em Ann. Appl. Probab.} {\bf 16} (1)  340-369. \MR{2209345}

 \bibitem{ZHCC11}
{\sc Zhang, L.-X.},  {\sc Hu, F.},  {\sc Cheung, S. H.} and  {Chan, W. S.} (2011).
Immigrated urn models-theoretical properties and applications. {\em Ann. Statist.}  {\bf 39} (1)  643-671. \MR{2797859}
\end{thebibliography}
 \end{document}